\documentclass[12pt]{amsart}
\usepackage{amssymb}
\usepackage{amsmath}
\numberwithin{equation}{section}
\usepackage{breakcites}
\usepackage{color}
\usepackage{graphicx}
\usepackage{epstopdf}

\newtheorem{thm}{Theorem}[section]
\newtheorem{lemma}[thm]{Lemma}
\newtheorem{cor}[thm]{Corollary}
\newtheorem{prop}[thm]{Proposition}
\newtheorem{con}[thm]{Conjecture}

\theoremstyle{definition}
\newtheorem{definition}[thm]{Definition}
\newtheorem{example}[thm]{Example}

\newtheorem{prob}[thm]{Problem}

\newtheorem{eg}[thm]{Example}

\theoremstyle{remark}
\newtheorem{remark}[thm]{Remark}

\renewcommand{\S}{\mathfrak S}
\newcommand{\s}{\sigma}

\newcommand\calC{{\mathcal{C}}}

\newcommand\C{{\mathbb{C}}}

\newcommand\Q{{\mathbb Q}}
\newcommand\Z{{\mathbb{Z}}}

\newcommand\PP{{\mathbb{P}}}
\newcommand\N{{\mathbb{N}}}

\newcommand\bq{\begin{equation}}
\newcommand\eq{\end{equation}}
\newcommand\beq{\begin{eqnarray*}}
\newcommand\eeq{\end{eqnarray*}}
\newcommand\ben{\begin{enumerate}}
\newcommand\een{\end{enumerate}}
\newcommand\bit{\begin{itemize}}
\newcommand\eit{\end{itemize}}

\newcommand\des{{\rm des}}
\newcommand\exc{{\rm exc}}
\newcommand\inv{{\rm inv}}
\newcommand\maj{{\rm maj}}

\newcommand\sg{{\mathfrak S}}
\newcommand\Des{{\rm DES}}

\newcommand\fix{{\rm fix}}
\newcommand\ch{{\rm ch}}

\newcommand\x{{\mathbf x}}

\newcommand\Par{{\rm Par}}

\newcommand\asc{{\rm asc}}

\newcommand\hgt{{\rm ht}}

\def\zz{{\mathbb Z}}
\def\nn{{\mathbb N}}

\def\cc{{\mathbb C}}
\def\pp{{\mathbb P}}

\def\cq{{\rm QSym}}
\def\ci{{\mathcal I}}

\def\xx{{\mathbf x}}

\def\inc{{\rm{inc}}}
\def\o{\bar {\mathbf o}}

\def\bft{s}
\def\hess{{\mathcal H}}

\def\Ps{{\rm \bf ps}}

\def\v{{\mathcal V}}

\newlength\cellsize 
\savebox2{%
\begin{picture}(15,15)
\put(0,0){\line(1,0){15}}
\put(0,0){\line(0,1){15}}
\put(15,0){\line(0,1){15}}
\put(0,15){\line(1,0){15}}
\end{picture}}
\newcommand\cellify[1]{\def\thearg{#1}\def\nothing{}%
\ifx\thearg\nothing
\vrule width0pt height\cellsize depth0pt\else
\hbox to 0pt{\usebox2\hss}\fi%
\vbox to 15\unitlength{
\vss
\hbox to 15\unitlength{\hss$#1$\hss}
\vss}}
\newcommand\tableau[1]{\vtop{\let\\=\cr
\setlength\baselineskip{-16000pt}
\setlength\lineskiplimit{16000pt}
\setlength\lineskip{0pt}
\halign{&\cellify{##}\cr#1\crcr}}}
\savebox3{%
\begin{picture}(15,15)
\put(0,0){\line(1,0){15}}
\put(0,0){\line(0,1){15}}
\put(15,0){\line(0,1){15}}
\put(0,15){\line(1,0){15}}
\end{picture}}
\newcommand\expath[1]{%
\hbox to 0pt{\usebox3\hss}%
\vbox to 15\unitlength{
\vss
\hbox to 15\unitlength{\hss$#1$\hss}
\vss}}

\begin{document}

\title[Chromatic quasisymmetric functions]
{ Chromatic quasisymmetric functions}
\author[Shareshian]{John Shareshian$^1$}
\address{Department of Mathematics, Washington University, St. Louis, MO 63130}
\thanks{$^{1}$Supported in part by NSF Grants
 DMS 0902142 and DMS 1202337}
\email{shareshi@math.wustl.edu}

\author[Wachs]{Michelle L. Wachs$^2$}
\address{Department of Mathematics, University of Miami, Coral Gables, FL 33124}
\email{wachs@math.miami.edu}
\thanks{$^{2}$Supported in part by NSF Grants
DMS 0902323 and DMS 1202755 and by   Simons Foundation Grant
\#267236.}

\begin{abstract} We introduce  a  quasisymmetric refinement of  Stanley's chromatic symmetric function.   We derive refinements of both Gasharov's Schur-basis expansion of the chromatic symmetric function and Chow's expansion in  Gessel's basis of fundamental quasisymmetric functions.  We present a conjectural refinement of Stanley's power sum basis expansion, which we prove in special cases.   We describe connections between the chromatic quasisymmetric function and both the $q$-Eulerian polynomials introduced in our earlier work and, conjecturally, representations of symmetric groups on cohomology of regular semisimple Hessenberg varieties, which have been studied by Tymoczko and others.   We  discuss an approach, using the results and conjectures herein, to the $e$-positivity conjecture of Stanley and Stembridge for incomparability graphs of $(3+1)$-free posets.  \end{abstract}

\date{June 30, 2014; Revised March 18, 2016}

\vspace*{-.2in}\maketitle

\vspace*{-.3in}\tableofcontents

\section{Introduction} \label{introsec}
We study a quasisymmetric refinement of Stanley's chromatic symmetric function.  We present refined results for our quasisymmetric functions, some proved herein and some conjectured, analogous to known results and conjectures of Chow, Gasharov, Stanley and Stanley-Stembridge on chromatic symmetric functions.  We present also a conjecture relating our work to work of Tymoczko and others on representations of symmetric groups on the cohomology of regular semisimple Hessenberg varieties.  Some of the results in this paper were presented, without proof, in our survey paper \cite{ShWa3}.  We assume throughout that the reader is familiar with basic properties of symmetric and quasisymmetric functions, as discussed in \cite[Chapter 7]{St2}.

 Let $G=(V,E)$ be a graph. Given a subset $S$ of the set $\pp$ of positive integers,   a {\em proper $S$-coloring} of $G$ is a function $\kappa:V\to S$ such that $\kappa(i) \neq \kappa(j)$ whenever $\{i,j\} \in E$.  
Let $\mathcal C(G)$ be the set of proper $\pp$-colorings of $G$.  In \cite{St3}, Stanley defined the {\em chromatic symmetric function} of $G$ as
\[
X_G(\x):=\sum_{\kappa \in \calC(G)}\x_\kappa,
\]
where $\x:=(x_1,x_2,\dots)$ is a sequence of commuting indeterminants and
$
\x_\kappa:=\prod_{v\in V} x_{\kappa(v)}
$.

It is straightforward to confirm that $X_G(\x)$ lies in the $\Q$-algebra $\Lambda_\Q$ of symmetric functions in $x_1,x_2,\ldots$ with rational coefficients.  The chromatic symmetric function gives more information about proper colorings than the well-studied chromatic polynomial $\chi_G:\pp \to \N$.   (Recall that $\chi_G(m)$ is the number of proper $\{1,2,\ldots,m\}$-colorings of $G$.)  Indeed,  $X_G(1^m) = \chi_G(m)$, where $X_G(1^m)$ is the specialization of $X_G(\x)$   obtained by setting $x_i=1$ for $1\le i \le m$ and $x_i=0$ for $i > m$.   Chromatic symmetric functions are studied in various papers, including \cite{St3,St4,Ga, Ga2, Ch0, Ch, wo,nw,MaMoWa,Hu,Gu}.

Some basic results on chromatic symmetric functions involve expansions in various bases for $\Lambda_\Q$.  The bases in question are indexed naturally by the set $\Par$ of all integer partitions.  If $B=(b_\lambda)_{\lambda \in \Par}$ is such a basis and the coefficients of $f \in \Lambda_{\Q}$ with respect to $B$ are all nonnegative, we say that $f$ is $b$-positive.  Bases of this type appearing herein are the elementary basis $(e_\lambda)_{\lambda \in \Par}$, the complete homogeneous basis $(h_\lambda)_{\lambda \in \Par}$, the Schur basis $(s_\lambda)_{\lambda \in \Par}$ and the power sum basis $(p_\lambda)_{\lambda \in \Par}$.

A main motivation for our study is to understand the conjecture of Stanley-Stembridge stated below and a related theorem of Gasharov discussed in Section \ref{schursec}.  Recall that the {\em incomparability graph} $\inc(P)$ of a poset $P$ has as its vertices the elements of $P$, with edges connecting pairs of incomparable elements, and that $P$
 is called $(r+s)${-\it free} if $P$ does not contain an induced subposet isomorphic to the direct sum of an $r$ element chain and an $s$ element chain.

 \begin{con} [Stanley-Stembridge conjecture {\cite[Conjecture 5.5]{StSte}}, {\cite[Conjecture 5.1]{St3}}] \label{stancon}   
Let $G=(V,E)$ be the incomparability graph of a $(3+1)$-free poset.  Then $X_G(\x)$ is $e$-positive.
\end{con}

A weaker result stating that $X_G(\x)$ is Schur-positive, when $G$ is as in the conjecture, follows from results of Haiman \cite{Ha}.  Gasharov \cite{Ga} proves this result by finding a formula for the coefficients of the Schur functions in the Schur-expansion of $X_G(\x)$.   Schur-positivity means that $X_G(\x)$ is the Frobenius characteristic of a representation of the symmetric group.  The Stanley-Stembridge conjecture asserts that this representation is a direct sum of representations induced from sign characters of Young subgroups.

Our refinement is as follows.

\begin{definition}\label{quasichromdef} Let $G=(V,E)$ be a graph whose vertex set $V$  is a finite subset of  $\pp$. The {\em chromatic quasisymmetric function} of  $G$ is
$$X_G(\x,t) := \sum _{\kappa \in \calC(G)} t^{\asc(\kappa)} \x_\kappa,$$
where 
$$\asc(\kappa) :=| \{\{i,j\} \in E : i < j \mbox{ and  } \kappa(i) < \kappa(j) \}|.$$
\end{definition}  

It is straightforward to confirm that $X_G(\x,t)$ lies in the ring $\cq_\Q[t]$.  That is, $X_G(\x,t)$ is a polynomial in $t$ with coefficients in the ring $\cq_\Q$ of quasisymmetric functions in $x_1,x_2,\ldots$ with rational coefficients.  (We can also think of $X_G(\x,t)$ as lying in the ring $\cq_{\Q[t]}$ of quasisymmetric functions with coefficients in the polynomial ring $\Q[t]$, but the first interpretation is more natural for our purposes.)

Note that the coefficients of $X_G(\x,t)$ need not be symmetric functions and that $X_G(\x,t)$ depends not only on the isomorphism type of $G$ but on the pair $(V,E)$.  Indeed, we will see in Example \ref{XGex}(a) that if $G$ is the path $1-3-2$ then the coefficients of $X_G(\x,t)$ are not symmetric, while in Example \ref{XGex}(b) we will see that if $G$ is the path $1-2-3$ then the coefficients of $X_G(\x,t)$ are indeed symmetric functions.

The graph $1-2-3$ belongs to a class that will be of most interest herein.  A {\it natural unit interval order} is obtained as follows.  Choose a finite set of closed intervals $[a_i,a_i+1]$ ($1 \leq i \leq n$) of length one on the real line, with $a_i<a_{i+1}$ for $1 \leq i \leq n-1$.  The associated natural unit interval order $P$ is the poset on $[n]:=\{1,\ldots,n\}$ in which $i<_P j$ if $a_i+1<a_j$.  A natural unit interval order is both $(3+1)$-free and $(2+2)$-free.  We show in Section \ref{symmsec} that if $G$ is the incomparability graph of a natural unit interval order then the coefficients of $X_G(\x,t)$ are symmetric functions and form a palindromic sequence.  This leads us to the following refinement of the unit interval order case of Conjecture~\ref{stancon}. (This special case of Conjecture~\ref{stancon}  is equivalent to a conjecture of Stembridge on immanants \cite[Conjecture 4.4]{Ste0}; see \cite[(5.1)]{StSte}.)

\begin{con} \label{introquasistan} Let $G$ be the incomparability graph of a natural unit interval order.   Then  the palindromic polynomial $X_{G}(\x,t)$ is $e$-positive and $e$-unimodal.  That is,  if $X_{G}(\x,t)  = \sum_{j=0}^m  a_j(\x) t^j$ then $a_j(\x)$ is $e$-positive for all $j$, and $a_{j+1}(\x) - a_j(\x)$ is $e$-positive whenever $0\le j <\frac {m-1} 2$.  
 \end{con} 
 
Conjecture \ref{introquasistan} applies to posets that are both $(3+1)$-free and $(2+2)$-free, while Conjecture \ref{stancon} applies to the much larger class of $(3+1)$-free posets.  However, Guay-Paquet shows in \cite{Gu} that if Conjecture \ref{stancon} holds for posets that are both $(3+1)$-free and $(2+2)$-free, then it holds for all $(3+1)$-free posets.  So, Conjecture \ref{introquasistan} implies Conjecture \ref{stancon}.

   In Section~\ref{schursec}, we prove the weaker result that each $a_j(\x)$ in Conjecture~\ref{introquasistan} is Schur-positive by  giving a formula for the coefficients of the Schur functions in the Schur-expansion of $a_j(\x)$, which refines Gasharov's formula.
 It follows that each $a_j(\x)$ is the Frobenius characteristic of a representation of the symmetric group $\S_n$.  As with Conjecture \ref{stancon}, our conjecture asserts that each such representation is the direct sum of representations induced from sign representations of Young subgroups.  Thus we can consider $X_G(\x,t)$ to be a graded (by $t$) representation of $\S_n$.

The palindromicity of  $X_G(\x,t)$, when $G$ is the incomparability graph of a natural unit interval order, suggests that $X_G(\x,t)$ might be the Frobenius characteristic of the representation of $\S_n$ on the cohomology of some manifold.  Conjecture \ref{introhesschrom} below says that this is essentially the case.

For any $n \in \pp$, there is a natural bijection between the set of natural unit interval orders on the set $[n]$ and the set of regular semisimple Hessenberg varieties of type $A_{n-1}$.  Such varieties were studied by De Mari and Shayman in \cite{DeSh}.  Regular semisimple Hessenberg varieties associated to arbitrary crystallographic root systems were studied by De Mari, Procesi and Shayman in \cite{DePrSh}.  A regular semisimple Hessenberg variety of type $A_{n-1}$ admits an action of an $n$-dimensional torus.  This action satisfies the conditions necessary to apply the theory given by Goresky, Kottwitz and MacPherson in \cite{GoKoMacP}, in order to determine the cohomology of the variety.  As observed by Tymoczko in \cite{Ty2}, the moment graph arising from this action admits an action of $\S_n$, and this determines a representation of $\S_n$ on the cohomology of the variety. This cohomology is concentrated in even dimensions.  

As is standard, we write $\omega$ for the involution on $\Lambda_\Q$ that exchanges $e_\lambda$ and $h_\lambda$.

\begin{con} \label{introhesschrom}\footnote{See Section~\ref{newsec}: {\em Recent developments}}
Let $P$ be a natural unit interval order with incomparability graph $G=(V,E)$, and let $\hess(P)$ be the associated regular semisimple Hessenberg variety.  Then \begin{equation} \label{mainconeq} \omega X_G(\x,t)=\sum_{j=0}^{|E|}\ch H^{2j}(\hess(P))t^j, \end{equation} where $\ch H^{2j}(\hess(P))$ is the Frobenius characteristic of the representation of $\S_n$ on the $2j^{th}$ cohomology group of $\hess(P)$ described in \cite{Ty2}.
\end{con}

The truth of Conjecture~\ref{introhesschrom} would give another proof of Schur-positivity of $X_G(\x,t)$.  In addition, combined with the hard Lefschetz theorem, it would establish Schur-unimodality, which is still open.
We have considerable evidence for Conjecture~\ref{introhesschrom}.  As shown in Sections~\ref{specsec} 
and~\ref{hesssec}, $H^{2j}(\hess(P))$ and the inverse Frobenius characteristic of the coefficient of $t^j$ in $\omega X_G(\x,t)$ have the same dimension.  Moreover, an interesting special case of the conjecture is known to be true. 

For each positive integer $n$, the path $G_n:=1-2- \cdots -n$ is the  incomparability graph of a natural unit interval order.  The chromatic quasisymmetric function $X_{G_n}(\x,t)$ has been studied in various guises and it is the motivating example behind our work.   It is the descent enumerator for Smirnov words of length $n$, that is, words over the alphabet $\pp$, with no equal adjacent letters.   This is the context in which the formula 
\begin{equation}\label{introcarlg} \sum_{n\ge 0} X_{G_n}(\x,t) z^n = {(1-t) E(z) \over E(zt) - t E(z)} ,\end{equation}
which  refines a  formula of Carlitz, Scoville, and Vaughan \cite{CaScVa}, was obtained by Stanley  (see {\cite[Theorem  7.2]{ShWa}}).  
(Here  $E(z) := \sum_{n \ge 0} e_n(\x) z^n$.)  It follows from (\ref{introcarlg}) and an obeservation of Haiman \cite{Ha} that 
$X_{G_n}(\x,t)$ is $e$-positive and $e$-unimodal.  

In another direction, Stanley \cite[Proposition 7.7]{St0} used a recurrence of Procesi \cite{Pr} to show that  the right hand side of (\ref{introcarlg}), with $E(z)$ replaced by $H(z):= \sum_{n\ge 0} h_n(\x) z^n$,  is equal to  the generating function for the Frobenius characteristic of the representation of $\S_n$ on the cohomology of the toric variety $\v_n$ associated with the Coxeter complex of type $A_{n-1}$.  As noted in \cite{DePrSh}, the regular semisimple Hessenberg variety associated to the natural unit interval order with incomparability graph $G_n$ is $\v_n$.  It follows that   
\begin{equation} \label{introsmirtor} \omega X_{G_n}(\x,t)= \sum_{j=0}^{n-1}  \ch H^{2j}(\v_n)t^j ,\end{equation} as predicted by Conjecture \ref{introhesschrom}. 

 For a graph $G=([n],E)$ and a permutation $\sigma \in \S_n$, define $\inv_G(\sigma)$ to be the number of edges $\{\sigma(i),\sigma(j)\}\in E$ such that $1 \leq i<j \leq n$ and $\sigma(i)>\sigma(j)$.  (So, $\inv_G$ generalizes the well-studied inversion statistic.)  It follows from results in \cite{DeSh} and \cite{DePrSh} (see \cite[Proposition 7.7]{St0}) that if natural unit interval order $P$ has incomparability graph $G$ then the Poincar\'e polynomial for $\hess(P)$ is
\begin{equation} \label {poinintro} \sum_{j=0}^{|E|}\dim H^{2j}(\hess(P)) t^j = A_G(t):= \sum_{\sigma \in \S_n}  t^{\inv_G(\sigma)}.\end{equation}
When $G=G_n$, $\inv_G(\sigma)$ is the inverse descent number $ \des(\sigma^{-1})$.  Hence   $A_{G_n}(t)$ is an Eulerian polynomial.  For this reason $A_G(t)$ is referred to as a generalized Eulerian polynomial in \cite{DeSh}.  In this paper, we  study a $q$-analog $A_G(q,t)$ of
$A_G(t)$ defined by pairing $\inv_G$ with a classical Mahonian permutation statistic of Rawlings \cite{Ra}.   We  prove that 
\begin{equation}\label{introspec} A_G(q,t) = (q;q)_n \Ps (\omega X_{G}(\x,t)) ,\end{equation}
where $\Ps$ denotes the stable principal specialization and $(p;q)_n = \prod_{i=0}^{n-1} (1-pq^i)$.
We use this and results in \cite{ShWa} to show that $A_{G_n}(q,t) $ equals the 
$q$-Eulerian polynomial $$A_n(q,t) := \sum_{\sigma \in \S_n} q^{\maj(\sigma)-\exc(\sigma)} t^{\exc(\sigma)}$$ introduced in \cite{ShWa1,ShWa}.

The paper is organized as follows.  In Section~\ref{basicsec} we discuss some basic properties of $X_G(\x,t)$.  While the easy results in  this Section hold for  arbitrary graphs on subsets of $\pp$, the results in the remainder of the paper concern  incomparability graphs of   posets.    In Section~\ref{fundquasec} we describe the expansion of $X_{G}(\x,t)$ in  Gessel's basis of fundamental quasisymmetric functions.   When $t$ is set equal to $1$, this reduces to Chow's \cite[Corollary 2]{Ch} expansion of $X_G(\x)$.  The expansion in the basis of fundamental quasisymmetric functions is used in Section~\ref{specsec} to obtain  (\ref{introspec}) and other results on $A_G(q,t)$.

In Section~\ref{symmsec} we give various characterizations of the natural unit interval orders and discuss some of their properties.  Here we prove that if $G$ is the incomparability graph of a natural unit interval order then $X_G(\x,t)$ is symmetric in $\x$. In the remaining sections of the paper,  $G$ is assumed to be the  incomparability graph of a natural unit interval order.
In Section~\ref{esec} we discuss Conjecture~\ref{introquasistan}.  

Section~\ref{schursec} contains the proof of our refinement of Gasharov's result giving the coefficients in the Schur basis expansion of $X_G(\x,t)$ in terms of a type of tableau called a $P$-tableau.  From this we prove that the coefficient of $s_{1^n}$ is 
$\prod_{i=1}^{n-1} [1+a_i]_t$, where  $a_i := |\{\{i,j\} \in E : i <j\}|$ and $[k]_t$ denotes the $t$-number $1+t+\dots+t^{k-1}$.  
Thus 
this coefficient is a palindromic, unimodal polynomial, as follows from  a standard tool discussed in Appendix~\ref{palin.appen}.

Expansion in the power sum basis is discussed in Section~\ref{powersec}.  We use our refinement of Gasharov's expansion, the Murnaghan-Nakayama rule, and combinatorial manipulations to prove that the coefficient of $\frac 1 n p_n$ in the power sum basis expansion  of $\omega X_G(\x,t)$ is $[n]_t \prod_{j=2}^n [b_j]_t$, where  $b_j = |\{ \{i,j\} \in E : i < j\}|$.  Since this coefficient is the same as the coefficient of $e_n$ in the $e$-basis expansion of $X_G(\x,t)$, this product formula gives some additional support for Conjecture~\ref{introquasistan}. Again we can conclude that the coefficient is a palindromic, unimodal polynomial.
We also present two versions of a conjectured combinatorial formula for the coefficient of $p_\lambda$ in the $p$-basis expansion for arbitrary partitions  $\lambda$.  It can be shown that upon setting $t=1$, this formula reduces to a result of Stanley in \cite{St3}.

In Section~\ref{calcsec}  we consider certain natural unit interval orders $P$ having very few pairs of comparable elements, and  obtain explicit formulae  for all the coefficients in the Schur basis and $e$-basis expansions of $X_{\inc(P)}$ as sums of products of $t$-numbers.  Once again, $e$-unimodality (and palindromicity) are evident from these formulas.

In Section~\ref{specsec} we examine specializations of $\omega X_G(\x,t)$.  The stable principal specialization yields, by (\ref{introspec}),   the generalized $q$-Eulerian polynomials $A_G(q,t)$ and  the nonstable principal specialization yields  the more general generalized $(q,p)$-Eulerian polynomials.  We observe that Conjecture~\ref{introquasistan} (or the weaker Schur-unimodality conjecture\footnote{See Section~\ref{newsec}: {\em Recent developments}}) implies the conjecture that $A_G(q,t)$ is $q$-unimodal, an assertion that is true for $G=G_n$.  The assertion is also true when we set $q=1$ by  (\ref{poinintro}) and the hard Lefschetz theorem.  We prove that  when we set $q$ equal to any  {\em primitive} $n$th root of unity $\xi_n$, $A_G(q,t)$ is a unimodal (and palindromic) polynomial in $t$ with nonnegative integer coefficients  by deriving an explicit formula for $A_G(\xi_n,t)$.  We conjecture that the same is true for any $n$th root of unity, which was shown to be true    for $G=G_n$ in \cite{SaShWa}.

In Section~\ref{hesssec} we describe the connection with Hessenberg varieties.  In 
particular Conjecture~\ref{introhesschrom} and its consequences are discussed. 

In the appendices we give background information  on permutation statistics, $q$-unimodality of the $q$-Eulerian   numbers, and $e$-unimodality of the Smirnov word enumerators. In Appendix~{\ref{app.eul} we derive an explicit  formula for the $q$-Eulerian numbers from which palindromicity and $q$-unimodality are evident. An analogous formula for the Smirnov word enumerator is also given.

\section{Basic properties and examples} \label{basicsec}
 All terms used but not defined in this paper are defined in \cite{St1} or \cite{St2}.   
For any ring $R$, let $\cq_R$ be the $R$-algebra of quasisymmetric functions in variables $x_1,x_2,\ldots$ with coefficients in  $R$ and let $\cq^n_R$ denote the $R$-submodule of $\cq_R$ consisting of homogeneous quasisymmetric functions of degree $n$.  

The following results are immediate consequences of  Definition~\ref{quasichromdef}.

\begin{prop} Let $G$ be a graph on a subset of $\pp$ of size $n$.  Then
$X_G(\x,t) \in \cq_\Z^n[t]$.
\end{prop}

\begin{prop} \label{disjointprop} Let $G$ and $H$ be graphs on disjoint finite subsets of $\pp$ and let $G+H$ be the disjoint union of $G$ and $H$.  Then
$$X_{G+H}(\x,t) = X_G(\x,t) X_H(\x,t).$$
\end{prop}

Suppose  $G=(V,E)$, where $V=\{v_1<\dots <v_n\}\subseteq \pp$.  Then  the isomorphic graph $G^\prime=([n],E^\prime)$, where 
$ E^\prime = \{\{i,j\} \in \binom{[n]}{2} : \{v_i,v_j\} \in E\}$, has the same chromatic quasisymmetric function as $G$; that is $$X_G(\x,t) = X_{G^{\prime}}(\x,t).$$ Hence, 
we may, without  loss of generality,  restrict our attention to graphs on vertex set $[n]$.

For $n \ge 0$, the $t$-analogs of $n$ and $n!$ are defined as, $$[n]_t := 1+t +t^2+ \dots +t^{n-1} \quad \mbox{ and } \quad [n]_t! := \prod_{i=1}^n [i]_t.$$

\begin{eg} Let $G=([n],\emptyset)$.  Since every map $\kappa:[n] \to \PP$ is a proper coloring, we have
\bq \label{pn1} X_{G}(\x,t)= e_1^n . \eq
\end{eg}
\begin{eg}  Let $G=([n],\binom{[n]}{2})$. Since $G$  is the complete graph, each proper coloring $\kappa$ is an injective map and there is a unique permutation $\sigma \in \S_n$ whose letters appear in the same relative order as $\kappa(n),\kappa(n-1),\dots,\kappa(1)$.  Clearly $\inv(\sigma) = \asc(\kappa)$.
  It follows from (\ref{majinveq}) that
\bq \label{pn2} X_{G}(\x,t)= e_n  \sum_{\sigma \in \S_n} t^{\inv(\sigma)} = [n]_t! \,e_n.\eq
\end{eg}

\begin{eg} \label{pathex} Let $G_n=([n],\{\{i,i+1\} : i \in [n-1]\})$.  Then  $G_n$ is the path $1-2-\dots-n$.
    To each proper $\pp$-coloring $\kappa$ of   $G_n$ one  can associate the word $w(\kappa):=\kappa(n),\kappa(n-1),\dots,\kappa(1)$.  Such words are characterized by the property that no adjacent letters are equal and  are sometimes referred to as {\em Smirnov words}.  Note that $\asc(\kappa)$ counts descents of $w(\kappa)$.
  The descent enumerator $W_n(\x,t)$ for Smirnov words is discussed in Appendix~\ref{smirsec}.  Since 
 \begin{equation} \label{smirgrapheq} X_{G_n}(\x,t) = W_n(\x,t),\end{equation} by Theorem~\ref{stanth} we have,
 \bq \label{pathchrom}1+  \sum_{n\ge 1}X_{G_n}(\x,t) z^n= {(1-t)E(z) \over E(zt) - tE(z)},\eq
where $E(z) = \sum_{n\ge 0} e_n(\x)z^n$.
\end{eg}

In the examples above, the  polynomial $X_G(\x,t)$ is symmetric in $\x$.  This is not always the case, see Example~\ref{XGex}(a).

Let $\alpha = (\alpha_1,\dots,\alpha_k)$  be a composition  and let $\alpha^{{\sf r}{\sf e}{\sf v}}$ denote the reverse composition $(\alpha_k,\dots,\alpha_1)$.  The monomial quasisymmetric function $M_\alpha(\x)$ is defined by
$$M_\alpha(\x) := \sum_{i_1 <\dots < i_k} x_{i_1}^{\alpha_1} \cdots x_{i_k}^{\alpha_k}.$$
Let  $\rho: \cq_\Z \to \cq_\Z$ be the involution defined on the basis of monomial quasisymmetric functions  by $\rho(M_\alpha) = M_{\alpha^{{\sf r}{\sf e}{\sf v}}}$ for each composition $\alpha$.  Note that every symmetric function is fixed by $\rho$. Extend the involution $\rho$ to $\cq_\Z[t]$.

\begin{prop} \label{rhoprop} For all graphs $G= (V,E)$, where $V$ is a finite subset of $\pp$,
$$\rho(X_G(\x,t)) = t^{|E|}X_G(\x,t^{-1}).$$
\end{prop}
\begin{proof}
  We construct an  involution $\gamma$ on $\calC(G)$.  For $\kappa \in \calC(G)$, let $\gamma(\kappa):V \to \pp$ be the coloring defined by  letting $\gamma(\kappa)(v) = m_1+m_2-\kappa(v)$,  for each $v \in V$, where $m_1 = \min_{u \in V} \kappa(u)$  and $m_2 = \max_{u \in V} \kappa(u)$.  It is clear that $\gamma(\kappa)$ is a proper coloring and that $\asc(\gamma(\kappa)) = |E| - \asc(\kappa)$.  This implies that the coefficient of $x_1^{\alpha_1} x_2^{\alpha_2} \dots x_m^{\alpha_m} t^j$ in $X_G(\x,t)$ equals the coefficient of $x_{m}^{\alpha_1} x_{m-1}^{\alpha_2} \dots x_1^{\alpha_m} t^{|E|-j}$ in $X_G(\x,t)$ for all $j \in \{0,1,\dots,|E|\}$ and compositions $\alpha:=(\alpha_1,\alpha_2,\dots, \alpha_m)$ of $|V|$.   
It follows that the coefficient of $M_\alpha t^j$  in the expansion of   $X_G(\x,t)$ in the basis $(M_\alpha t^j)_{\alpha,j}$ equals the coefficient of $M_{\alpha^{{\sf r}{\sf e}{\sf v}}} t^{|E|-j}$ in the expansion of $X_G(\x,t)$, which equals the coefficient of $M_{\alpha} t^{|E|-j}$ in the expansion of $\rho(X_G(\x,t))$.  Hence the desired result holds.
\end{proof}

\begin{cor}For all graphs $G= (V,E)$, where $V$ is a finite subset of $\pp$, we have
$$\rho(X_G(\x,t)) = \sum_{\kappa \in \mathcal C(G)} t^{\des(\kappa)} \x_\kappa,$$
where 
$$\des(\kappa) :=| \{\{u,v\} \in E : u < v \mbox{ and  } \kappa(u) > \kappa(v) \}|.$$
Consequently if $X_G(\x,t)\in \Lambda_\Z[t]$ then
\begin{equation}\label{deseq} X_G(\x,t) = \sum_{\kappa \in \mathcal C(G)} t^{\des(\kappa)} \x_\kappa. \end{equation}
 \end{cor}

\begin{cor} \label{sympalincor} If $X_G(\x,t) \in \Lambda_\Z[t]$ then $X_G(\x,t)$ is palindromic in $t$ with center of symmetry $\frac {|E|}{2}$, in the sense that $X_G(\x,t) = t^{|E|} X_G(\x,t^{-1})$ (see Definition~\ref{paldef}).
\end{cor}

\section{Expansion in the fundamental basis} \label{fundquasec}

From now on we assume that a given graph $G$ is the incomparability graph $\inc(P)$ of a poset $P$ on some finite subset of $\pp$.  In this section we refine Chow's \cite[Corollary 2]{Ch} expansion of $X_G(\x)$ in Gessel's basis of fundamental quasisymmetric functions.  Expansion in the basis of fundamental quasisymmetric functions is a useful tool for obtaining results about permutation statistics via specialization.   This is discussed in Section~\ref{specsec}.

For $n \in \pp$ and $S \subseteq [n-1]$, let $D(S)$ be the set of all functions $f:[n] \rightarrow \pp$ such that
\begin{itemize}
\item $f(i) \geq f(i+1)$ for all $i \in [n-1]$, and \item $f(i)>f(i+1)$ for all $i \in S$.
\end{itemize}
The  {\em fundamental quasisymmetric function} associated with $S \subseteq [n]$ is defined as\footnote{This is a nonstandard version of Gessel's fundamental quasisymmetric function.  Our $F_{n,S}$ is equal to $L_{\alpha(S)}$ in \cite{St2}, where $\alpha(S)$ is the reverse of the composition associated with $S$.}
\[
F_{n,S}:=\sum_{f \in D(S)}\x_f,
\]
where $\x_f := x_{f(1)}x_{f(2)} \cdots x_{f(n)}$.
It is straightforward to confirm that $F_{n,S} \in \cq_{\Z}^n$.  In fact (see \cite[Proposition 7.19.1]{St2}),  $\{F_{n,S}:S \subseteq [n-1]\}$ is a basis for $\cq^n_{\Z}$.

For a graph $G:=([n],E)$, a $G$-inversion of a permutation $\sigma \in \sg_n$ is an edge $\{\sigma(i),\sigma(j)\} \in E$ such that $i < j$ and $\sigma(i) > \sigma(j)$.  Let $\inv_G(\sigma)$ be the number of $G$-inversions of $\sigma$; that is
\begin{equation} \label{invdef}  \inv_G(\sigma):=|\{ \{\sigma(i),\sigma(j) \} \in E: 1\le i<j \le n, \,\,\sigma(i) > \sigma(j)  \} |.
 \end{equation}
For a poset $P$ on $[n]$ and $\sigma \in \sg_n$, define the $P$-descent  set of $\sigma$
as
\begin{equation} \label{desdef}
 \Des_P(\sigma) :=  \{i \in [n-1] : \sigma(i) >_P \sigma(i+1)\}.
 \end{equation}
  Clearly, when $G$ is the complete graph, $\inv_G$ is the usual $\inv$ statistic and when $P$ is the chain $1<_P \dots <_P n$, $\Des_P$ is the usual descent set $\Des$ (see Appendix~\ref{qEulersec}).

We define $\omega$ to be the involution on $\cq_\Z$ that maps  $F_{n,S}$ to $F_{n,[n-1] \setminus S}$ for each $n \in \nn$ and $S \subseteq [n-1]$.

Note that the restriction of $\omega$ to $\Lambda_\Z$ is the involution mapping $h_n$ to $e_n$ for all $n \in \nn$.

\begin{thm} \label{qchow} Let  $G$ be the incomparability graph of a  poset $P$ on $[n]$.   Then 
$$\omega X_G(\x,t) = \sum_{\sigma \in \S_n} t^{\inv_G(\sigma)} F_{n,\Des_{P}(\sigma)}.$$
\end{thm}

\begin{proof} Our proof  follows the same path as the proof of Corollary~2 in \cite{Ch}. We use the terminology {\it sequencing} and {\it  labeling} from \cite{Ch}.   Both a sequencing and a labeling of $G$ are bijections from $[n] $ to $ [n]$, but for a sequencing, the codomain $[n]$ is viewed as the vertex set of $G$, and for a labeling, the domain $[n]$ is viewed as the vertex set of $G$.

Let $O(G) $ be the set of acyclic orientations of $G$.  For $\o \in O(G)$, let $\asc(\o)$ be the number of directed edges 
$(i,j)$ of $\o$ for which $i<j$, and $\calC(\o)$ be the set of proper colorings $\kappa: [n] \to \pp$ of $G$ that are compatible with the orientation $\o$ in the sense that  $\kappa(i) < \kappa(j)$ whenever $(i,j)$ is a directed edge of $\o$. We have, 
\begin{equation} \label{orient}  X_G(\x,t) = \sum_{\o\in O(G)} t^{\asc(\o)} \sum_{\kappa \in \calC(\o)} \x_\kappa,\end{equation}
since each proper coloring $\kappa$ of $G$ is compatible with a unique acyclic orientation $\o$  of $G$, and $\asc(\kappa) = \asc(\o)$. 

We view each acyclic orientation $\o$ of $G$  as a poset on $[n]$ by taking the  transitive closure of the relations given by the directed edges of $\o$.   For any labeling $\alpha$  of $G$,  
let $\mathcal L(\o, \alpha)$ denote the set of all linear extensions of the labeled poset $(\o,\alpha)$;  that is, $\mathcal L(\o, \alpha)$ is the set of permutations $[\alpha(v_1),\dots,\alpha(v_n)]$ (written in one line notation) such that  if $v_i <_{\o} v_j$ then $i <j$.  

For each $\o\in O(G)$, choose a decreasing labeling $w_{\o}$ of  the poset $\o$, that is, choose  $w_{ \o}$ so that $w_{\o}(x) > w_{\o}(y) $ if $x <_{\o} y$.  Given a subset $S \subseteq [n-1]$, define $n-S := \{i : n-i \in S\}$.
  It follows from Stanley's theory of $P$-partitions (see \cite[Corollary 7.19.5]{St2}) that
\begin{equation} \label{ppart} \sum_{\kappa \in \calC(\o)} \x_\kappa = \sum_{\sigma \in \mathcal L(\o, w_{\o})} F_{n,n-{\Des(\sigma)}}.\end{equation}

Let $e$ be the identity labeling of $G$ (or $\o$).  Each element of $\mathcal L(\o,e)$ is a sequencing of $G$.  For $s \in  \mathcal L(\o,e)$, let $w_{\o} s$ denote the product of $w_{\o} $ and $s$ in the symmetric group $\sg_n$. We have $\sigma \in \mathcal L(\o, w_{\o})$ if and only if $\sigma = w_{\o} s$ for some sequencing  $s\in \mathcal L(\o,  e) $. Hence  (\ref{ppart}) becomes
\begin{equation} \label{eqnew}  \sum_{\kappa \in \calC(\o)} \x_\kappa  = \sum_{s \in \mathcal L(\o, e)} F_{n,n-{\Des(w_{\o} s)}}. \end{equation}   
Combining (\ref{eqnew}) and  (\ref{orient}), we obtain
\begin{equation} \label{eqchromascorient} X_G(\x,t) = \sum_{\o\in O(G)} t^{\asc(\o)} \sum_{s \in \mathcal L(\o, e)} F_{n,n-{\Des(w_{\o} s)}}.\end{equation}
Each sequencing $s$ of $G$ determines a unique acyclic orientation $\o$ of $G$ for which $s\in \mathcal L(\o,  e) $.  Denote this acyclic orientation by $\o(s)$.  We can therefore rewrite (\ref{eqchromascorient}) as
$$ X_G(\x,t) = \sum_{s} t^{\asc(\o(s))} F_{n,n-{\Des(w_{\o(s)} s)}},$$ where $s$ ranges over all sequencings of $G$.
Note that $$ \asc( \o(s)) = \inv_G(s^{{\sf r}{\sf e}{\sf v}}),$$ where $s^{{\sf r}{\sf e}{\sf v}}$ denotes the reverse of the sequence $s$.  Hence
\begin{equation} \label{chromdes} X_G(\x,t) = \sum_{s\in \S_n} t^{\inv_G(s^{{\sf r}{\sf e}{\sf v}})} F_{n,n-{\Des(w_{\o(s)} s)}}.\end{equation}

Up to this point, $w_{\o}$ was an arbitrary decreasing labeling of the poset $\o$. Now we need to use a specific decreasing labeling.   For each acyclic orientation $\o$ of $G$, we construct a decreasing labeling $\widetilde{w_{\o}}$ of $\o$ as follows:    Since all maximal elements of $\o$ are comparable in $P$ there will always be an element among these elements that is larger than the others under $P$.  Label this element with $1$.   Assuming that $k$ elements have already been labeled, label the $P$-largest element in the  set of maximal  unlabeled elements of $\o$   with $k+1$.    Do this for $k=1$ to $n-1$.  

{\em Claim:} 
For all $x,y$ that are incomparable in $\o$, if $x<_P y$ then $\widetilde{w_{\o}}(x) >\widetilde{w_{\o}}(y) $.  

Suppose the claim is false.  Then there are  incomparable $x$ and $y$ in  $\o$ such that $ x <_P y$ and $\widetilde{w_{\o}}(x) <\widetilde{w_{\o}}(y) $.  Consider the step in the  construction of $\widetilde{w_{\o}}$ in which $x$ is labeled.  Since $x <_P y$ and $x$ is labeled ahead of   $y$, there must be an unlabeled element $z$ such that $z>_{\o} y$  and $z <_P x$.  Choose an $\o$-minimal such $z$. Let $$y<_{\o} y_1<_{\o} y_2 <_{\o} \cdots <_{\o}y_k= z$$ be an unrefinable chain of the poset $\o$. 

Since $z$ is unlabeled,  the $y_i$'s are unlabeled.  By the minimality of $z$ we have that $$y_i \not<_P x$$ for each $i\in [k-1]$. 
Since $x$ is labeled before each 
 $y_i$, we have $x\not<_{\o} y_i$.  Since $x$ and $y$ are incomparable in $\o$ and $y <_{\o} y_i$, 
 we have $y_i \not <_{\o} x$.  Hence $x$ and $y_i$ are incomparable in $\o$ for each $i$. This implies 
 that there is no edge of $G$ joining $x$ to any of the $y_i$'s.  It follows that $x$ is comparable to each $y_i$  in  $P$.  In particular $x$ and $y_{k-1}$ are comparable in $P$.  Since $y_{k-1} \not <_P x$, we have $x \le_P y_{k-1}$.  Combining this with $z<_P x$ yields $z<_P y_{k-1}$, which contradicts the fact that  $z$ and $y_{k-1}$ are joined by an edge in $G$.  Therefore we have established the claim.
 
Next we use the claim to show that for any sequencing $s$ of $G$,
\begin{equation}\label{desdeseq} \Des_P(s) = [n-1] \setminus \Des(\widetilde{w_{\o(s)}}s).\end{equation}
 Let $i \in \Des_P(s)$.  This means that $s(i) >_P s(i+1)$. Clearly $s(i)$ and $s(i+1)$ are incomparable in $\o(s)$ since $s \in \mathcal L(\o(s),e)$.  Thus  the claim implies that $\widetilde{w_{\o(s)}}(s(i)) < \widetilde{w_{\o(s)}}(s(i+1)) $, which means that $i \in [n-1] \setminus \Des(\widetilde{w_{\o(s)}}s)$.  Hence, $\Des_P(s) \subseteq [n-1] \setminus \Des(\widetilde{w_{\o(s)}}s)$.  Conversely, if $i \in [n-1] \setminus \Des(\widetilde{w_{\o(s)}}s)$ then $\widetilde{w_{\o(s)}}(s(i)) < \widetilde{w_{\o(s)}}(s(i+1)) $.  Since $\widetilde{w_{\o(s)}}$ is a decreasing labeling,  $s(i) \not\le_{\o(s)} s(i+1) $.
Since $s \in \mathcal L(\o(s),e)$, we also have $s(i+1) \not\le_{\o(s)} s(i) $.  Hence $s(i)$ and $s(i+1)$ are incomparable in $\o(s)$.  Therefore the claim can again be applied  yielding $i \in \Des_P(s)$.  Hence $[n-1] \setminus \Des(\widetilde{w_{\o(s)}}s) \subseteq \Des_P(s)$.  We  conclude that (\ref{desdeseq}) holds.

By (\ref{chromdes}) and (\ref{desdeseq})  we have,
$$\omega  X_G(\x,t) =\sum_{s \in \sg_n} t^{\inv_G(s^{{\sf r}{\sf e}{\sf v}}) } F_{n,n-{\Des_P(s)}}.$$
For all $\sigma \in \sg_n$, we have $n-{\Des_P(\sigma)}= \Des_{P^*}(\sigma^{{\sf r}{\sf e}{\sf v}})$,
where  $P^*$ denotes the dual of $P$.
  It follows that
$$\omega  X_G(\x,t) =\sum_{\s \in \sg_n} t^{\inv_G(\sigma) } F_{n,\Des_{P^*} (\s)}.$$
The result now follows from $\inc(P^*) = \inc(P)$.
\end{proof}

\begin{example}[a]  \label{XGex} Let $G$  be the graph with vertex set $[3]$ and edge set $\{\{1,3\},\{2,3\}\}$.  In other words $G$ is the path $1-3-2$.  Then $G= \inc(P)$, where $P$ is the poset on $[3]$ whose only order relation is $ 1<_P2$.  We have
\begin{center}\begin{tabular}{|c|c|c|}
\hline{\color{red} $\sg_3$}  & {\color{red}$\inv_G$} & {\color{red}$\Des_P$} \\ \hline\hline 
$123$ &$0$ & $\emptyset$ \\   \hline 
$132$ &$1$ & $\emptyset$   \\ \hline 
$213$ &$0$ & $\{1\}$ 
\\ \hline 
$231$ &$1$ & $\emptyset$ 
\\ \hline
$312$ &$2$ & $\emptyset$ 
\\ \hline
$321$ &$2$ & $\{2\}$  \\
 \hline 
\end{tabular} \end{center}

\vspace{.1in}
Hence by Theorem~\ref{qchow}  
\begin{eqnarray*}\omega X_{G}(\x,t)&=& (F_{\emptyset} + F_{3,\{1\}}) + 2 F_{\emptyset} t + (F_{\emptyset} + F_{3,\{2\}})t^2 \\ &=&  (h_3 + F_{3,\{1\}}) + 2 h_3 t + (h_3 + F_{3,\{2\}})t^2.\end{eqnarray*}
Note that  $X_{G}(\x,t)$ is  not symmetric in $\x$ and that $X_{G}(\x,t)$ is not  palindromic as a polynomial in $t$; cf. Corollary~\ref{sympalincor}.  

(b).  Now let $G$  be the graph with vertex set $[3]$ and edge set $\{\{1,2\},\{2,3\}\}$.  In other words $G$ is the path $1-2-3$.  Then $G= \inc(P)$ where $P$ is the poset on $[3]$ whose only order relation is $ 1<_P3$.  We have
\begin{center}\begin{tabular}{|c|c|c|}
\hline{\color{red} $\sg_3$}  & {\color{red}$\inv_G$} & {\color{red}$\Des_P$} \\ \hline\hline 
$123$ &$0$ & $\emptyset$ \\   \hline 
$132$ &$1$ & $\emptyset$  \\ \hline 
$213$ &$1$ & $\emptyset$ 
\\ \hline 
$231$ &$1$ & $\{2\}$ 
\\ \hline
$312$ &$1$ & $\{1\}$ 
\\ \hline
$321$ &$2$ & $\emptyset$  \\
 \hline 
\end{tabular} \end{center}

\vspace{.1in}
Hence by Theorem~\ref{qchow}  
\begin{eqnarray*}\omega X_{G}(\x,t)&=&  F_{3,\emptyset} + (2 F_{3,\emptyset} +  F_{3,\{1\}} + F_{3,\{2\}}) t + F_{3,\emptyset}  t^2 \\ &=&  h_3 + (h_3 +  h_{2,1}) t + h_3 t^2.\end{eqnarray*}  
Note that  $X_G(\x,t)$ is symmetric in $\x$ and that $X_G(\x,t)$ is a palindromic polynomial in $t$; cf. Corollary~\ref{sympalincor}.
\end{example}

\section{Natural unit interval orders} \label{symmsec}
In this section we describe a general class of posets $P$ for which $X_{\inc(P)}(\x,t)$ is symmetric in $\x$.

A {\em unit interval order} is a  poset  that is isomorphic to a finite collection $\ci$ of  intervals $[a,a+1]$ on the real line, partially ordered by the relation $[a,a+1]<_\ci[b,b+1]$ if $a+1<b$.  A well known characterization of the unit interval orders is that they are $(2+2)$-free and $(3+1)$-free  (see \cite{ScSu}).  The isomorphism classes of unit interval orders form a Catalan class, see \cite[Exercise 6.19 ddd]{St2}.

 Define a {\em natural unit interval order} to be a poset $P$ on   a finite subset of $\pp$ that satisfies both conditions
\begin{enumerate} 
\item $x<_P y$ implies $x<y$ in the natural order on $\pp$, and
\item if the direct sum $\{x<_P z\} + \{y\}$ is an induced subposet of $P$ then  $x<y<z$ in the natural order on $\pp$.
\end{enumerate}

We provide now a useful characterization of natural unit interval orders and show that our formal definition is equivalent to the informal definition of natural unit interval orders given in the introduction.

Given a sequence ${\bf m}:=(m_1,\dots,m_{n-1})$ of positive integers satisfying
\begin{itemize}
\item[(a)] $m_1 \le m_2 \le \dots \le m_{n-1} \le n$, and
\item[(b)] $m_i \ge i$ for all $i$,
\end{itemize}
let $P({\bf m})$ denote the poset on $[n]$ with order relation given by $i<_{P({\bf m})} j$ if  $i < n$ and $j \in \{m_i+1,m_{i}+2,\dots,n\}$.
(It is straightforward to confirm that the given relation does indeed make $P({\bf m})$ a poset.)

\begin{prop} \label{natunitprop2}
Let $P$ be a poset on $[n]$.  The following conditions on $P$ are equivalent. \begin{itemize} \item[(A)] $P$ is a natural unit interval order. \item[(B)] $P=P({\bf m})$ for some ${\bf m}=(m_1,\ldots,m_{n-1})$ satisfying conditions (a) and (b) above. \item[(C)] There exist $n$ real numbers $y_1<\ldots<y_n$ such that, for 
$i,j \in [n]$, $y_i+1<y_j$ if and only if $i<_P j$. \end{itemize} \end{prop}

\begin{proof}
(A)$\implies$ (B). Suppose that $P$ is a natural unit interval order on $[n]$.   For each $i \in [n-1]$, let $m_i = \max \{j \in [n] : j \not>_P i\}$.  Since $i \not>_P i$, we have $m_i \ge i$.  Let $x <y \in [n-1]$.  We need to show $m_x \le m_y$.
Suppose $m_x > m_y$.  Then by construction of $m_y$, we have $y <_P m_x$.  We claim $x$ is  comparable to neither $y$ nor $m_x$.  If $x$ is comparable to $y$ then $x <_Py$ by condition (1) of the definition of natural unit interval order.  
By transitivity this implies $x <_P m_x$, which is impossible.  Hence $x$ is not comparable to $y$.  If $x$ is comparable to $m_x$ then $x=m_x$, which implies $y <_P x$, contradicting condition (1).  Now since $x$ is  comparable to neither $y$ nor $m_x$, by condition (2) of the definition of natural unit interval order we have $y <x < m_x$, which is a contradiction.  We conclude that $m_x \le m_y$.  We have shown that ${\bf m} = (m_1,\dots,m_{n-1}) $ satisfies the conditions (a) and (b).
  
To show $P = P({\bf m})$, we only need to show that for all $x \in [n-1]$, if $y$ is such that $x < y < m_x$ then $x \not<_P y$.  Suppose $x <_P y$.  Then by condition (2) of the definition of  natural unit interval order, $m_x$ is comparable to $x$ or $y$.  Since $m_x$ is not comparable to $x$, it must be comparable to $y$.  By condition (1), $m_x >_P y$.  By transitivity $m_x >_P x$, which is impossible.  Hence our assumption $x <_P y$ is false.

(B) $\implies$ (C). Now suppose that $P=P({\bf m})$ for some ${\bf m}$ satisfying conditions (a) and (b).  We prove by induction on $n$ that $P$ satisfies condition (C), the base case $n=1$ being trivial.  Assume that $n>1$.  Let ${\bf m^\prime}=(m_1^\prime,\ldots,m^\prime_{n-2})$, where $m^\prime_i=\min\{n-1,m_i\}$ for all $i \in [n-2]$. Note that the subposet of $P$ induced on $[n-1]$ is $P({\bf m^\prime})$.
By inductive hypothesis, we can find $y_1<\ldots<y_{n-1}$ such that $y_i+1<y_j$ if and only if $i<_{P({\bf m^\prime})} j$.  Let us choose and fix such $y_1,\ldots,y_{n-1}$.  We may assume that there do not exist $i,j$ with $y_j-y_i=1$.  Indeed, if such $i,j$ exist, we may subtract some small number from $y_k$ for each $k \geq j$ to obtain another sequence satisfying the desired conditions.

If $m_1=n$, then $P$ is an antichain.  We can ignore our original choices and set $y_i=\frac{i}{n}$ in this case.  If $m_{n-1}=n-1$ then $i<_P n$ for all $i \in [n-1]$.  In this case, we can choose $y_n=y_{n-1}+2$.  If $m_1<m_{n-1}=n$, find the smallest $i$ such that $m_i=n$.  Then $i$ and $j$ are incomparable in $P$ whenever $i<j<n$, as are $j$ and $n$. In particular, $y_i+1>y_{n-1}$.  On the other hand, $j<_P n$ for all $j \in [i-1]$.  In this case, we pick $y_n$ strictly between $\max\{y_{i-1}+1,y_{n-1}\}$ and $y_i+1$.

(C) $\implies$ (A). Assume that $P$ satisfies (C).  If $i<_P j$ then $y_i<y_i+1<y_j$ and $i<j$ in the natural order.  If $i<_P k$ and neither $i$ nor $k$ is comparable to $j$ in $P$ then $y_i+1<y_k$ and $y_j \in [y_i-1,y_i+1] \cap [y_k-1,y_k+1]$.  It follows that $y_i<y_j<y_k$ and $i<j<k$ in the natural order.  Therefore, $P$ is a natural unit interval order. \end{proof}

It follows from the equivalence of (A) and (C) in Theorem \ref{natunitprop2} that every natural unit interval order is a unit interval order and that every unit interval order is isomorphic with some natural unit interval order.  As mentioned above, the number of isomorphism classes of unit interval orders on $n$ elements is the Catalan number $C_n:=\frac{1}{n+1}{{2n} \choose {n}}$.  The number of sequences of length $n-1$ satisfying conditions (a) and (b) above is also $C_n$.  Indeed, this follows quickly from \cite[Exercise~6.19~s]{St2}.
Thus we have the following result.

\begin{prop} \label{natunitprop1}
Every unit interval order is isomorphic with a unique natural unit interval order.
\end{prop}

\begin{eg}  \label{Gex} For each $r \in [n]$, define $P_{n,r}$ to be the poset on  $[n]$ with order relation given by $i <_{P_{n,r}} j$ if $j-i \ge r$.  The poset $P_{n,r}$ is an example of a natural unit interval order.  Indeed using the notation of  Proposition~\ref{natunitprop2}, $$P_{n,r} = P(r,r+1,r+2,\dots,n,\dots,n).$$   The incomparabilty graph $G_{n,r}$ of $P_{n,r}$ is the graph  with vertex set $[n]$ and edge set $\{\{i,j\} : 0 < j-i < r\}$.  Note that  $G_{n,2}$ is the path $G_n:=1-2-\dots-n$.  This class of examples will appear throughout the paper.
\end{eg}

We show now that chromatic quasisymmetric functions of incomparability graphs of natural unit interval orders are symmetric functions.

\begin{lemma} \label{pathcolorlem} Let $G$ be the incomparability graph of a natural unit interval order $P$ and let $\kappa$ be a proper $\pp$-coloring of $G$.  For every  integer $a$, each connected component of the  induced subgraph $G_{\kappa,a}$ of $G$ consisting of vertices colored by $a$ or $a+1$, is a path of the form $i_1-i_2-\cdots-i_j$, where $i_1<\dots<i_j$. 
\end{lemma}

\begin{proof}
Note first that $G_{\kappa,a}$ is bipartite and therefore contains no cycle of length three.  If $\{x,y\}$ and $\{y,z\}$ are edges in $G$ and $\{x,z\}$ is not then condition (2) in the definition of  natural unit interval order implies that either  $x<y<z$ or $z<y<x$.  It follows that if $i_1-i_2-\cdots-i_j$ is a path of  $G_{\kappa,a}$ then either $i_1< i_2 < \dots < i_j$ or $i_1>i_2 > \dots > i_j$, which implies that $G_{\kappa,a}$ contains no cycle.
So, every connected component of $G_{\kappa,a}$ is a tree.  We claim that no vertex of $G_{\kappa,a}$ has degree larger than two, from which it follows that every connected component of $G_{\kappa,a}$ is a path with increasing vertices, as claimed.  Indeed, assume for contradiction that $x,y,z$ are all neighbors of $w$ in $G_{\kappa,a}$.  Suppose $x <w$.  Then 
$x<w<y$ and $x<w<z$ since $x-w-y$ and $x-w-z$ are paths in  $G_{\kappa,a}$.  But  we cannot have $w < y,z$ since $y-w- z$ is also a path in $G_{\kappa,a}$.
\end{proof}

\begin{thm}  \label{symprop} Let $G$ be the incomparability graph of a natural unit interval order.   Then $$X_{G}(\x,t) \in \Lambda_{\zz}[t].$$
\end{thm}

\begin{proof}  For each $a \in \pp$, we construct an involution $\psi_a$ on the set $\calC(G)$ of proper colorings of $G$ that, for each proper coloring, exchanges the number of occurrences of the color  $a$ with the number of occurrences of the color $a+1$, and preserves the statistic $\asc$.   From this we may conclude that $X_{G}(\x,t)$ is symmetric in $\x$.

By Lemma~\ref{pathcolorlem}, for $\kappa\in \calC(G)$, each connected component of  $G_{\kappa,a}$ is a path 
$i_1-i_2 - \cdots - i_j$ with $i_1 <\dots < i_j$.  Let $\psi_a(\kappa)$ be the coloring of $G$ obtained from $\kappa$ 
by replacing each color $a$ by $a+1$ and each $a+1$ by $a$ in a connected component $i_1-i_2-\cdots-i_j$ 
of $G_{\kappa,a}$ whenever $j$ is odd.    (If $j$ is even leave the coloring of the component unchanged.) It is easy 
to see that  $\psi_a(\kappa)$ is a proper coloring.  Moreover,  the number of occurrences of the color $a$ in $\psi_a(\kappa)$ is equal 
to the number of occurrences of $a+1$ in $\kappa$, and the number of  occurrences of the color $a+1$  in $\psi_a(\kappa)$ is equal to the 
number of occurrences of $a$ in $\kappa$.  It is also clear that $\asc(\kappa) = \asc(\psi_a(\kappa))$ on each 
component, since $i_1 <\dots < i_j$. \end{proof}
 
By Corollary~\ref{sympalincor} we have,
\begin{cor} \label{genpalcor} If $G$ is the incomparability graph of a natural unit interval order then $X_G(\x,t)$ is palindromic as  a polynomial in $t$.  \end{cor}

\begin{remark} The converse of Theorem~\ref{symprop} does not hold.    It would be interesting to find a  characterization of those graphs $G$ for which $X_G(\x,t)$ is symmetric in $\x$.
\end{remark}

\section{Expansion in the elementary basis} \label{esec}

Recall the graph $G_{n,r}$ given in Example~\ref{Gex}.   
When $r=1$, $G_{n,r}$ is the empty graph whose chromatic quasisymmetric function is $e_1^n$. When $r=2$, $G_{n,r}$ is the path $G_n$ of Example~\ref{pathex}.  Hence  by (\ref{smirgrapheq}), a closed form formula for $X_{G_{n,2}}(\x,t)$ is given in Theorem~\ref{newsymth}.  Formulae for some other values of $r$ are derived in Section~\ref{calcsec}.  For general $r$, we have only a formula for the coefficient of $e_n$ in the $e$-basis expansion, which is given in Corollary~\ref{ecoefcor}.  
Table~1 gives a summary of these formulae.    
  The formulae get increasingly complex and difficult to compute as $r$ decreases from $n$ down  to $3$.

\begin{table}[htdp]
\hspace*{-.2in}
\begin{tabular}{|@{}c@{}||l|l|l|l|l|}
\hline{\color{blue} \scriptsize{$r$}} & {\color{blue}\scriptsize{$X_{G_{n,r}}({\bf x},t)$} }  
\\
\hline
\hline{\color{blue} \scriptsize{1}} & \scriptsize{{$e_1^n$}}  \\

\hline{\color{blue} \scriptsize{2}} & \scriptsize{$\displaystyle{{\sum_{m=1}^{\lfloor {n+1 \over 2} \rfloor}\,\,\sum_{\substack{
k_1,\dots, k_m \ge 2 \\ \sum k_i = n+1}}\,\,\,{e_{(k_1-1,k_2, \dots, k_m)} }\,\, {\color{red}t^{m-1} \!\!  \prod_{i=1}^m  [k_i-1]_{t} }}}$} \\
\hline {\color{blue} \vdots} & \\ 
\hline{\color{blue} \scriptsize{\,\,$r$\,\,}} & \scriptsize{$e_n{\color{red} [n]_t  [ r-1]_t^{n-r} [r-1]_t!} \qquad+\qquad$ ???}\\
\hline {\color{blue} \vdots} &\\ 
\hline{\color{blue} \scriptsize{\,\,$n-2$\,\,}} & \scriptsize{$e_n{\color{red} [n]_t  [ n-3]_t^2 [n-3]_t!}  +  e_{(n-1,1)} \,\,{\color{red}  t^{n-3}[n-4]_t!([2]_t[n-3]_t^2+[n-1]_t[n-4]_t)} $  } \\ & \scriptsize{$  + e_{(n-2,2)} {\color{red} t^{2n-7}[2]_t[n-4]_t! } $} \\
\hline{\color{blue} \scriptsize{\,\,$n-1$\,\,}} &\scriptsize{$e_n {\color{red}[n]_t  [n-2]_t [n-2]_t !  } + e_{(n-1,1)} {\color{red} t^{n-2} [n-2]_t! }$}\\
\hline{\color{blue} \scriptsize{$n$}} & \scriptsize{$e_n{\color{red} [n]_t! } $}
\\
\hline \end{tabular}
\caption{Expansion in the $e$-basis.}
\end{table}

See Definition~\ref{paldef} for the definition  of $b$-positivity and $b$-unimodality of polynomials in $\Lambda_\Q[t]$, where $b$ is a basis for $\Lambda_\Q$.  
By Proposition~\ref{tooluni}, we see that for each formula of Table 1, the coefficient of  each $e_\lambda$  is a positive, unimodal and palindromic polynomial in $t$ with the same center of symmetry.   Hence by Proposition~\ref{tooluni2}  we conclude that the  conjectured
 refinement of the  Stanley-Stembridge conjecture, which we restate now,  is true for $G=G_{n,r}$  when $ r \le 2$ and $r \ge n-2$.  We have also verified the conjecture by computer for $G=G_{n,r}$ for $1 \le r \le n \le 8$.

 \begin{con} \label{quasistan} Let $G$ be the incomparability graph of a natural unit interval order.   Then $X_{G}(\x,t)$  is an $e$-positive and $e$-unimodal polynomial in $t$. \end{con}

We have the following easy consequence of Proposition~\ref{disjointprop} and a  generalized version of Proposition~\ref{tooluni} for coefficient ring $\Lambda_\Z$, which reduces Conjecture~\ref{quasistan} to the case of connected graphs.
\begin{prop}Let $G$ and $H$ be graphs on disjoint finite subsets of $\pp$. 
Then
$X_{G+H}(\x,t)$ is symmetric in $\x$, $e$-positive,  $e$-unimodal, and palindromic if  $X_G(\x,t)$ and $X_H(\x,t)$ have these properties.
\end{prop}

Let $\Par(n,j)$ be the set of partitions of $n$ into $j$ parts and let $c^G_\lambda$ be the coefficient of $e_\lambda$   in the $e$-basis expansion of the chromatic symmetric function $X_G(\x)$. In \cite{St3} Stanley proves that for any graph $G$ whose vertex set has size $n$, the number of acyclic orientations of $G$ with $j$ sinks is equal to  $\sum_{\lambda \in \Par(n,j)}  c^G_\lambda$. 
The following refinement of this result  provides a bit of further evidence for $e$-positivity of $X_G(\x,t)$.  The proof is an easy modification of Stanley's proof and is omitted.

\begin{thm} \label{acyclicth} Let $G$ be the incomparability graph of   a natural unit interval order of size $n$.   For each $\lambda \vdash n$, let $c^G_{\lambda}(t)$ be the coefficient of $e_\lambda$ in  the $e$-basis expansion of
$X_{G}(\x,t)$.  Then 
$$\sum_{\lambda \in \Par(n,j)}  c^G_\lambda(t) = \sum_{o \in {\mathcal O}(G,j)} t^{\asc(o)},$$
where ${\mathcal O}(G,j)$ is the set of acyclic orientations of $G$ with $j$ sinks and $\asc(o)$ is   the number of directed edges $(a,b)$  of $o$ for which $a<b$.  
\end{thm}

Since there is only one partition in $\Par(n,1)$, Theorem~\ref{acyclicth} gives a combinatorial description of 
the coefficient $c_{n}^G(t)$ of $e_n$ in the $e$-basis expansion of $X_G(\x,t)$.  In Corollary~\ref{ecoefcor}  below we give  a closed form formula for the coefficient.

\section{Expansion  in the Schur basis}  \label{schursec} In this section we refine the unit interval order case of Gasharov's Schur-positivity result \cite{Ga} and use the refinement to obtain a closed form formula for the coefficient of $s_{1^n}$ in the Schur basis expansion of $X_G(\x,t)$. 

\begin{definition}[Gasharov \cite{Ga}] \label{gashdef} Let $P$ be a poset of size $n$ and  $\lambda$ be a partition of $n$.   A {\em $P$-tableau  of shape $\lambda$} is  a filling  of a Young diagram of shape $\lambda$ (in English notation)
with elements of $P$ such that
\begin{itemize}
  \item each element of $P$ appears exactly once,
  \item if $y \in P$ appears immediately to the right of $x \in P$ then $y>_P x$, 
  \item if $y \in P$  appears immediately below  $x\in P$ then $y \not <_P x$.
\end{itemize}
\end{definition}

Given a finite poset $P$ on a subset of $\pp$, let ${\mathcal T}_P$ be the set of all $P$-tableaux. 
 For  $T \in \mathcal T_P$ and $G=\inc(P)$, define a {\em $G$-inversion} of $T$ to be  an edge $\{i,j\} \in E(G)$ such that 
 $i<j$ and $i$ appears below  $j$ in $T$ (not necessarily in the same column).  Let $\inv_G(T)$ be the number of $G$-inversions of $T$
  and let  $\lambda(T)$ be the shape of $T$.
   \begin{example} Let  $G=G_{9,3}$ and let $P:=P_{9,3}$.  Then 
$$T=\tableau{ {2} & {6}  & {9}    \\ {1} & {4} & {8} \\  {3} & {7}  \\ {5} }$$
 is a $P$-tableau of shape $(3,3,2,1)$ and 
$$\inv_G(T) = |\{ \{1,2\},  \{3,4\}, \{4,6\}, \{5,6\},\{5,7\},  \{7,8\} , \{7,9\},\{8,9\} \}| = 8 .$$
\end{example}

\begin{thm}\label{schurcon}
Let $G$ be the incomparability graph of  a natural unit interval order $P$.  Then
\[
X_G(\x,t)=\sum_{T \in {\mathcal T}_P}t^{\inv_G(T)}s_{\lambda(T)}.
\]  Consequently $X_G(\x,t)$ is Schur-positive.
\end{thm}

\begin{remark}Schur-unimodality of $X_G(\x,t)$, which is implied by Conjecture~\ref{quasistan}, is still 
open.\footnote{See Section~\ref{newsec}: {\em Recent developments}}  It is also consequence of Conjecture~\ref{hesschrom}; see Proposition~\ref{uniconsprop}.
\end{remark}

\begin{example} Let $G:=G_{3,2}$ and $P:=P_{3,2}$. The $P$-tableaux are as follows
$$ \tableau{ {1} & {3} \\  {2} } \qquad \tableau{{1}\\{2}\\{3}} \qquad \tableau{{1}\\{3}\\{2}} \qquad  \tableau{{2}\\{1}\\{3}}
\qquad  \tableau{{3}\\{2}\\{1}}$$
For the first, third, and fourth tableau, $\inv_G(T) = 1$; for the second $\inv_G(T) = 0$, and for the last $\inv_G(T) = 2$.  Hence Theorem~\ref{schurcon} yields, 
$$ X_G(\x,t) = t s_{(2,1)} + (1+2t + t^2) s_{1^3},$$
which is consistent with  Example~\ref{XGex} (b).\end{example}

\begin{remark}
Our proof of Theorem \ref{schurcon} follows closely the proof of the $t=1$ version of this theorem 
due to  Gasharov \cite{Ga}.   Gasharov's proof involves a sign-reversing involution.  While Gasharov's involution does not preserve the $G$-inversion number, we
find a modified involution that does so.
\end{remark}

For the proof of Theorem~\ref{schurcon}, we need  a more general notion than $P$-tableau.
\begin{definition}[Gasharov \cite{Ga}] Let $P$ be a poset of size $n$ and  $\alpha=(\alpha_1,\ldots,\alpha_n)$ be a weak composition of $n$.
A {\it $P$-array} of shape $\alpha$ is an $n \times n$ matrix $(a_{ij})$ such that  each of the following conditions hold\begin{itemize} \item each $a_{ij}$ lies in $P\uplus \{0\}$ \item for each $x \in P$, there is exactly one $a_{ij}$ equal to $x$  \item if $j>1$ and $a_{ij} \neq 0$, then $a_{i,j-1} \neq 0$ and $a_{i,j-1}<_Pa_{ij}$  \item for each $i \in [n]$, we have $|\{j \in [n]:a_{ij} \neq 0\}|=\alpha_i$. \end{itemize}   
\end{definition}

Less formally, if we define the Young diagram of a weak composition in a manner analogous with the definition of the Young diagram of a partition, then (after removing zeroes), a $P$-array of shape $\alpha$ is a filling of the Young diagram of $\alpha$ with the elements of $P$ such that each nonempty row of the diagram is filled with a chain from $P$ written in increasing order.  Note if a $P$-array $A$ of shape $\alpha$ satisfies the additional condition
\begin{itemize} \item  if $i>1$ and $a_{ij} \neq 0$ then $a_{i-1,j} \neq 0$ and $a_{i-1,j} \not>_P a_{ij}$, \end{itemize}   then $\alpha$ is a partition of $n$ and $A$ is a $P$-tableau of shape $\alpha$.

Given a finite poset $P$ on a subset of $\pp$ with $G=\inc(P)$, the notion of $\inv_G$ extends to $P$-arrays in a natural way.  For a $P$-array $A= (a_{ij})$, define $\inv_G(A) $ to be the number of pairs $(a_{ij},a_{st})$ such that each of the following holds 
\begin{itemize} \item $i<s$, \item $a_{ij}>a_{st}>0$ (in the natural order on $\pp$) \item $\{a_{ij},a_{st}\}\in E(G)$. \end{itemize}

\begin{example} Let  $G:=G_{8,3}$ and let $P:=P_{8,3}$.  Then 
$$A=\tableau{ {5} & {8}     \\ {1} & {4} & {7} \\ \hspace{.2in}{\color{white}{\blacksquare}\over \blacksquare} 
\hspace{-.1in}{\color{white}\blacksquare}  \\ {2} & {6}  \\ {3} }$$
 is a $P$-array of shape $(2,3,0,2,1)$ and 
 $$\inv_G(A) = |\{\{2,4\} ,\{3,4\},\{3,5\},  \{4,5\}, \{6,7\}, \{6,8\}, \{7,8\} \}| = 7.$$
\end{example}

\begin{proof}[Proof of Theorem~\ref{schurcon}]  Write $$X_G(\x,t)=\sum_{\lambda \in \Par(n)}c_\lambda(t)s_\lambda,$$ where $\Par(n)$ is the set of partitions of $n$.  For $\lambda \in \Par(n)$, let 
${\mathcal T}_{P,\lambda}$ be the set of $P$-tableaux of shape $\lambda$.  The claim of the theorem is that 
\begin{equation} \label{clam} c_\lambda(t)=\sum_{T \in {\mathcal T}_{P,\lambda}}t^{\inv_G(T)}. \end{equation}  To 
prove that (\ref{clam}) holds, we use the standard inner product $\langle \cdot,\cdot \rangle$ on the ring of 
symmetric functions, as defined in \cite[Section 7.9]{St2}.  
Pad the partition $\lambda$ with zeros (if necessary) so that it has length $n$, so $\lambda=(\lambda_1 \ge \ldots \ge \lambda_n)$, where $\lambda_n \ge 0$.  For $\pi \in \S_n$, let 
$\pi(\lambda) \in \zz^n$ be the $n$-tuple whose $i^{th}$ entry  is given by $$\pi(\lambda)_i:=\lambda_{\pi(i)}-\pi(i)+i.$$  For any 
$\alpha=(\alpha_1,\ldots,\alpha_n) \in \zz^n$, set $h_\alpha=\prod_{i=1}^{n }h_{\alpha_i}$, where $h_m=0$ if $m<0$ and $h_0 =1$.  
So, $h_\alpha$ is a symmetric function.   Let $\epsilon$ be the sign character of $\S_n$.  It follows from the Jacobi-Trudi indentity (see \cite[Theorem 7.16.1]{St2}) that \begin{eqnarray*} c_\lambda(t) & = & \langle X_G(\x,t),s_
\lambda \rangle \\ & = & \sum_{\pi \in \S_n}\epsilon(\pi) \langle X_G(\x,t),h_{\pi(\lambda)} \rangle .\end{eqnarray*}  It 
follows from the fact that $\langle m_\mu,h_\lambda \rangle=\delta_{\mu,\lambda}$ that $\langle 
X_G(\x,t),h_{\pi(\lambda)} \rangle$ is the coefficient of $\prod_{i=1}^{n} x_i^{\pi(\lambda)_i}$ in $X_G(\x,t)$.  Hence by  (\ref{deseq}),
$$c_\lambda(t) = \sum_{\pi \in \sg_n} \epsilon(\pi) \sum_{\kappa\in \calC_{\pi(\lambda)}(G)} t^{\des(\kappa)},$$ where
$\calC_{\pi(\lambda)}(G)$ is the set of proper $[n]$-colorings of $G$ with $\pi(\lambda)_i$ occurrences of the color $i$ 
for each $i \in [n]$.

For each proper $[n]$-coloring $\kappa$ of $G$, we define $A(\kappa)$ to be the unique $n\times n$ $P$-array whose $i^{th}$ row contains the elements of $\kappa^{-1}(i)$, for each $i \in [n]$.  Note that 
\begin{eqnarray}\label{desinveq} \des(\kappa)=\inv_G(A(\kappa)).\end{eqnarray}  Note also that for each $n\times n$ $P$-array $A$, there is a unique proper $[n]$-coloring $\kappa$ of $G$ with $A(\kappa)=A$.

Let ${\mathcal A}_{P,\lambda}$ be the set of all pairs $(A,\pi)$ such that \begin{itemize} \item $\pi \in \S_n$
 \item $A$ is a $P$-array of shape $\pi(\lambda)$. \end{itemize}  We see now that \begin{equation} \label{arr} c_\lambda(t)=\sum_{(A,\pi) \in {\mathcal A}_{P,\lambda}}\epsilon(\pi)t^{\inv_G(A)}. \end{equation}
Set $${\mathcal B}_{P,\lambda}:=\{(A,\pi) \in {\mathcal A}_{P,\lambda}: A \not\in {\mathcal T}_{P,\lambda}\}.$$ 
For any nonidentity $\pi \in \S_n$, there is some inversion $(\pi(i),\pi(j))$ of $\pi$, such that $\pi(\lambda)_i<\pi(\lambda)_j$.  Therefore, $\pi(\lambda)$ is not a partition and $(A,\pi) \in {\mathcal B}_{P,\lambda}$ for all $P$-arrays $A$ of shape $\pi(\lambda)$.  Therefore, (\ref{clam}) will follow from (\ref{arr}) once we produce an involution $\psi$ on ${\mathcal B}_{P,\lambda}$, mapping $(A,\pi)$ to $(A^\prime,\pi^\prime)$, such that \begin{itemize} \item $\inv_G(A^\prime)=\inv_G(A)$, and \item $\epsilon(\pi^\prime)=-\epsilon(\pi)$. \end{itemize} Note that, as $\epsilon(\pi^\prime) \neq \epsilon(\pi)$, $\psi$ is fixed-point-free, as is necessary for our purposes.

Let $(A,\pi) \in {\mathcal B}_{P,\lambda}$.  
We call $a_{ij}$ ``bad" if $i>1$, $a_{ij} \neq 0$, and either $a_{i-1,j}=0$ or $a_{i-1,j}>_Pa_{ij}$.  Since $A$ is not a $P$-tableau, there exists some bad $a_{ij}$.  Let $c=c(A)$ be the smallest $j$ such that there is a bad $a_{ij}$ and, having found $c$, let $r=r(A)$ be the largest $i$ such that $a_{i+1,c}$ is bad.  We will obtain a $P$-array $A^\prime$ from $A$ by moving certain elements of row $r$ lying weakly to the right of column $c$ down to row $r+1$ and moving certain elements of row $r+1$ lying strictly to the right of column $c$ to row $r$.  Let us describe this process precisely.

Define $$C_r(A):=\{a_{rj}:c \leq j \leq |\kappa^{-1}(r)|\}$$ and $$C_{r+1}(A):=\{a_{r+1,j}:c<j \leq |\kappa^{-1}(r+1)|\},$$ 
where $\kappa$ is the unique proper $[n]$-coloring such that $A(\kappa) = A$.
Let $H(A)$ be the subgraph of $G$ induced on $C_r(A) \bigcup C_{r+1}(A)$.
For $i \in \{r,r+1\}$ set $$O_i(A):=\{x \in C_i(A):x \mbox{ lies in a connected component of odd size in $H(A)$}\}$$ and $$E_i(A):=C_i(A) \setminus O_i(A).$$
Also, set $$I_r(A):=\{a_{rj}:j<c\}$$ and $$I_{r+1}(A):=\{a_{r+1,j}:j \leq c\}.$$ Note that if $c>1$, then, as there is no bad $a_{ij}$ with $j<c$, we have  $a_{r,c-1} \neq 0$. 

If $x \in I_r(A)$ and $y \in E_r(A)$ then $x<_Py$, as $A$ is a $P$-array.  If $x \in O_{r+1}(A)$ and $y \in E_r(A)$ then $x,y$ are comparable in $P$, as $x$ and $y$ live in different connected components of $H(A)$.  We claim now that if $x \in I_r(A)$ and $y \in O_{r+1}(A)$ then $x<_Py$.  To prove this claim, it suffices to show that if $c>1$ and $a_{r+1,c+1} \neq 0$ then $a_{r,c-1}<_Pa_{r+1,c+1}$.  Under the given conditions, $$a_{r+1,c-1}<_Pa_{r+1,c}<_Pa_{r+1,c+1}.$$  By the definition of $c$, $a_{r,c-1}\not>_Pa_{r+1,c-1}$.  Since $P$ is $(3+1)$-free, it follows that $a_{r,c-1}<_Pa_{r+1,c+1}$ as desired.  Therefore, $I_r(A) \cup E_r(A) \cup O_{r+1}(A)$ is a chain in $P$ in which $I_r(A)$ is an initial chain.

If $x \in I_{r+1}(A)$ and $y \in E_{r+1}(A)$ then $x<_Py$.  If $x \in O_r(A)$ and $y \in E_{r+1}(A)$ then $x$ and $y$ lie in different components of $H(A)$ and are therefore comparable in $P$.  If $x \in I_{r+1}(A)$ and $y \in O_r(A)$ then $$x \leq_P a_{r+1,c}<_Pa_{r,c} \leq_Py.$$ Therefore, $I_{r+1}(A) \cup E_{r+1}(A) \cup O_r(A)$ is a chain in $P$ in which $I_{r+1}(A)$ is an initial chain.

Writing $A_i$ for the $i^{th}$ row of $A$ we can now define the $P$-array $A^\prime=(a^\prime_{ij})$, whose $i^{th}$ row is $A^\prime_i$, to be the unique $P$-array satisfying \begin{itemize} \item if $i \in [n] \setminus \{r,r+1\}$ then $A^\prime_i=A_i$, \item the set of nonzero elements in $A^\prime_r$ is $I_r(A) \cup E_r(A) \cup O_{r+1}(A)$, and \item the set of nonzero elements in $A^\prime_{r+1}$ is $I_{r+1}(A) \cup E_{r+1}(A) \cup O_r(A)$. \end{itemize} 

If $i \in [n]$ and $j \in [c-1]$ then $a^\prime_{ij}=a_{ij}$.  Moreover, $a^\prime_{r+1,c}=a_{r+1,c}$ and $$a^\prime_{rc} \in E_r(A) \cup O_{r+1}(A) \cup \{0\}.$$  Certainly $y>_Pa_{r+1,c}$ if $y \in O_{r+1}(A)$.  If $y \in E_r(A)$ then $$y \geq_P a_{rc}>_Pa_{r+1,c}.$$  Therefore, $a^\prime_{r+1,c}$ is bad in $A^\prime$ and $c(A^\prime)=c$.  Now $r(A^\prime)=r$, since $a^\prime_{ic}=a_{ic}$ for $i>r $.  Also, $E_r(A^\prime)=E_r(A)$, $O_r(A^\prime)=O_{r+1}(A)$, $E_{r+1}(A^\prime)=E_{r+1}(A)$ and $O_{r+1}(A^\prime)=O_r(A)$.  It follows that  $(A^\prime)^\prime=A$. 

For $i \in [n]$, the row $A_i$ has $\pi(\lambda)_i$ nonzero entries.  For $i \in \{r,r+1\}$, let $\rho_i$ be the number of nonzero entries in $A^\prime_i$. By Lemma~\ref{pathcolorlem}, since $H(A)$ is an induced subgraph of $G_{\kappa,r}$ (where $\kappa$ is the unique proper $[n]$-coloring such that $A(\kappa) = A$), its connected components are paths 
$i_1 - \cdots - i_j$ where $i_1< \dots < i_j$.  Hence each connected component of even size in $H(A)$ has exactly half of its elements in $A_r$.  It follows that each $\rho_i$ is the number of nonzero elements in the $i^{th}$ row of the array obtained from $A$ by moving all of $C_r(A)$ to row $r+1$ and moving all of  $C_{r+1}(A)$ to row $r$. 
 Therefore, \begin{equation} \label{rhor} \rho_r=c-1+\pi(\lambda)_{r+1}-c=\pi(\lambda)_{r+1}-1, \end{equation} and  \begin{equation} \label{rhor1} \rho_{r+1}=c+\pi(\lambda)_r-(c-1)=\pi(\lambda)_r+1. \end{equation}  Let $\tau_r \in \S_n$ be the transposition exchanging $r$ and $r+1$ and set $\pi^\prime=\pi\tau_r$.   It follows from (\ref{rhor}) and (\ref{rhor1}) that $A^\prime$ is a $P$-array of shape $\pi^\prime(\lambda)$.  Moreover, $(\pi^\prime)^\prime=\pi$ and $\epsilon(\pi^\prime)=-\epsilon(\pi)$.  So, the map $\psi$ sending $(A,\pi)$ to $(A^\prime,\pi^\prime)$ is a sign-reversing involution on ${\mathcal B}_{P,\lambda}$.

It remains to show that $\inv_G(A^\prime)=\inv_G(A)$.  By (\ref{desinveq}), this is equivalent to showing $\des(\kappa) = \des(\kappa^\prime)$, where $\kappa$ (resp. $\kappa^\prime$) is the unique proper $[n]$-coloring of $G$ 
that satisfies $A(\kappa) = A$ (resp. $A(\kappa^\prime) = A^\prime$).   Note that $\kappa^\prime$ is obtained from $\kappa$ by replacing each occurrence of the color $r$ by $r+1$ and each occurrence of  $r+1$ by $r$ in each odd size connected component of $H(A)$.  
Recall that the connected components of $H(A)$ are  paths 
with numerically increasing vertices. 
When $\kappa$ and $\kappa^\prime$ are restricted to such a path of odd size, the number of descents of $\kappa$  equals the number of ascents of $\kappa$, which in turn equals the number of descents of $\kappa^\prime$.   It follows  that $\des(\kappa) = \des(\kappa^\prime)$.
\end{proof}

 From Theorem \ref{schurcon}, we are able to derive a nice closed form formula for the coefficient  of $s_{1^n}$ in the Schur basis expansion of   $X_{G}(\x,t)$.

 \begin{thm} \label{s1n}
Let $G$ be the incomparability graph of a natural unit interval order $P$ of size $n$. Then \begin{equation} \label{eschureq} \sum_{\scriptsize\begin{array}{c} T \in \mathcal T_P \\ \lambda(T) = 1^n \end{array}}t^{\inv_G(T)}=\prod_{j=1}^{n-1}[1+a_j]_t,\end{equation}
where 
$a_j:= |\{  \{j,i\} \in E(G): j < i\}|.$  Consequently   the coefficient  of $s_{1^n}$ in the Schur basis expansion of  $X_{G}(\x,t)$ is equal to the right hand side of (\ref{eschureq})
  and is therefore positive, palindromic and unimodal with center of symmetry $\frac{|E(G)|}{2}$.
\end{thm}

\begin{proof} 
  We may assume without loss of generality that the vertex set of $G$ is $[n]$.  Clearly there is a $G$-inversion preserving bijection between the set of $P$-tableaux of shape $1^n$ and the set $$\mathcal {ND}_P:=\{\sigma\in \sg_n: \Des_P(\sigma) = \emptyset\},$$  where $\Des_P(\sigma)$ is defined in (\ref{desdef}) and $G$-inversions of permutations are defined in (\ref{invdef}).  Indeed, by reading the entries of a $P$-tableau $T$ of shape $1^n$ from top down, one gets a permutation $\sigma$ in $\mathcal {ND}_P$ with $\inv_G(\sigma) = \inv_G(T)$.  Hence 
$$\sum_{\scriptsize\begin{array}{c} T \in \mathcal T_P \\ \lambda(T) = 1^n \end{array}}t^{\inv_G(T)}=\sum_{\sigma \in \mathcal{ND}_P} t^{\inv_G(\sigma)} .$$

To complete the proof  of (\ref{eschureq}) we need only show that for all $k \in [n]$,
$$\sum_{\sigma \in \mathcal{ND}_{P,k}} t^{\inv_G(\sigma)} = \prod_{j=k}^{n-1}[1+a_j]_t,$$
where 
 $\mathcal{ND}_{P,k}$ is the set of permutations of $\{k,\dots,n\}$ with no $P$-desccents (here we are viewing permutations of $\{k,\dots,n\}$ as words and defining $\inv_G$ and $P$-descents as in  Definition~\ref{leftrightdef}). 
This is equivalent to the   recurrence,
 \begin{equation} \label{receq} \sum_{\sigma \in \mathcal{ND}_{P,k}} t^{\inv_G(\sigma)} = [1+a_k]_t \sum_{\sigma \in \mathcal{ND}_{P,k+1}} t^{\inv_G(\sigma)},\end{equation} for all $k \in [n-1]$.
To prove this recurrence we  show that  each word in $\mathcal {ND}_{P,k}$ can be obtained by inserting $k$ into a word in $\mathcal {ND}_{P,k+1}$ in one of $1+a_k$ allowable places.
 
First we observe that removal of $k$ from any $\sigma\in \mathcal {ND}_{P,k}$ does not introduce a $P$-descent.
 Indeed, suppose a $P$-descent is introduced. Let  $s$ be the letter to the immediate left of $k$ in $\sigma$ and let $t$ be the letter to the immediate right of $k$.  Then 
$s >_P t$ and $s \not >_P k$.  It follows that $k$ is not comparable to $s$ or $t$.  Since $k < s,t$, this contradicts the fact that $P$ is a natural unit interval order.  Hence no $P$ descent was introduced by removal of $k$ from $\sigma$.  This means that each word in $ \mathcal {ND}_{P,k}$ can be obtained by inserting $k$ in an allowable position of a word in $ \mathcal {ND}_{P,k+1}$.

Now let $A_k=  \{i: k < i \le n,  \{k,i\} \in E(G)\}$; so that $|A_k| = a_k$.  Let $\alpha \in \mathcal{ND}_{P,k+1}$. The 
letters of $A_k$ form a subword $\alpha_{i_1},\alpha_{i_2},\dots, \alpha_{i_{a_k}}$ of $\alpha$.  To avoid introducing a $P$-descent when inserting $k$ into  $\alpha$,  the insertion must be made at the beginning of  $\alpha$ or just to the right of one of the 
$\alpha_{i_j}$'s.  If $k$ is inserted at the beginning of $\alpha$ then no $G$-inversions are introduced.   If  $k$ is 
inserted in $\alpha$ to  the immediate right of $\alpha_{i_j}$ for some $j \in [a_k]$, then $j$ $G$-inversions are 
introduced. 
Hence the recurrence (\ref{receq}) holds.

The consequence follows from Theorem~\ref{schurcon} and by applying Proposition~\ref{tooluni} to the product in (\ref{eschureq}).
\end{proof}

\section{Expansion in the power sum basis}  \label{powersec}
Conjecture~\ref{quasistan} implies  that $\omega X_G(\x,t)$ is $p$-positive when $G$ is the incomparability graph of a natural unit interval order.  In this section we give a conjectural formula for the coefficients  in the $p$-basis expansion, which can be shown to reduce to Stanley's p-basis expansion when $t=1$ (see \cite{St3}).  The conjectural formula  for the coefficient of $p_\lambda$ reduces to the simple closed form formula given in the following result when $\lambda = (n)$

\begin{thm} \label{pcoefth} Let $G$ be the incomparablity graph  of   a natural unit interval order of size $n$.   Then the coefficient $ c(t)$ of $\frac 1 n p_{n}$ in the $p$-basis expansion of $\omega X_{G}(\x,t)$ is given by
  \begin{equation} \label{pcoefeq1} c(t) = [n]_t \prod_{i=2}^n [b_i]_t,\end{equation}
where 
\begin{equation} \label{bi.eq} b_i = |\{ \{j,i\} \in E(G) : j < i\}|. \end{equation}  
Consequently the coefficient  is positive, palindromic, and unimodal with center of symmetry $\frac{|E(G)|}{2}$.
  \end{thm}
  
  \begin{proof}  
   Equation~(\ref{pcoefeq1}) follows immediately from Lemmas~\ref{pcoeflem2} and~\ref{pcoeflem3} below. The consequence follows from Proposition~\ref{tooluni}.
  \end{proof}
  
Since  the coefficient of $\frac 1 n p_n$ in the $p$-basis expansion of any symmetric function  is the same as the coefficient of $h_n$ in the $h$-basis expansion, we have the following corollary.
  
  \begin{cor} \label{ecoefcor} Let $G$ be the incomparability graph  of   a natural unit interval order of size $n$.  Then the coefficient of $e_n$ in the $e$-basis expansion of expansion of $X_{G}(\x,t)$ is equal to $[n]_t \prod_{i=2}^n [b_i]_t$ and is therefore positive, palindromic, and unimodal with center of symmetry $\frac{|E(G)|}{2}$.
  \end{cor}
  
\begin{definition}  \label{leftrightdef} Let $P$ be a poset on a finite subset of $\pp$ of size $n$ and let $\alpha=(\alpha_1,\dots,\alpha_m)$ be a word over alphabet $P$.     We  say that $\alpha$ has   a  left-to-right $P$-maximum at $r\in [n]$ if $\alpha_r >_P \alpha_s $ for all $ s \in [r-1]$.  The  left-to-right $P$-maximum at $1$ will be referred to as   trivial.  We say that $i \in [m]$ is a $P$-descent of $\alpha$ if $\alpha_i >_P \alpha_{i+1}$.  
  For  $G=\inc(P)$, let
 $$\inv_G(\alpha) := |\{\{\alpha_i, \alpha_j\} \in E : i<j \mbox{ and } \alpha_i > \alpha_j\}|.$$ 
  \end{definition}
  
 For a poset $P$ on a finite subset of $\pp$, let $\mathcal N_P$ denote the set of permutations of $P$ (viewed as words over alphabet $P$) with no $P$-descents and no  
  nontrivial  left-to-right $P$-maxima.

 \begin{lemma} \label{pcoeflem2} Let $G$ be the incomparability graph of a natural unit interval order $P$ of size $n$ and let $c(t)$ be the coefficient of $\frac 1 n p_n$ in the $p$-basis expansion of $\omega X_G(\x,t)$.  Then 
 \begin{equation} \label{pcoefeq2} c(t) =   \sum_{\sigma \in \mathcal N_P} t^{\inv_{G}(\sigma)}.\end{equation}
\end{lemma}
\begin{lemma} \label{pcoeflem3} Let $G$ be the incomparability graph of a natural unit interval order $P$ of size $n$ and let $b_i$ be as in Theorem~\ref{pcoefth}.  Then
\begin{equation} \label{pcoefeq3}    \sum_{\sigma \in \mathcal N_P} t^{\inv_{G}(\sigma)} = [n]_t \prod_{i=2}^n [b_i]_t .\end{equation}
 \end{lemma}

 \begin{proof}[Proof of Lemma \ref{pcoeflem2}]
We use the Murnaghan-Nakayama rule and Theorem~\ref{schurcon} to obtain an expansion of  $\omega X_G(\x,t)$ in the $p$-basis.
Recall that a {\it rim-hook tableau} of shape $\lambda \vdash n$ and type $\nu = (\nu_1,\dots,\nu_k) \vdash n$ is a filling of the Young diagram of shape $\lambda$ with positive integers  so that
\begin{itemize}
\item each  $i\in [k]$ appears $\nu_i$ times,
\item every row and every column is weakly increasing, and
\item for each $i\in [k]$, the set of squares occupied by $i$ forms a rim-hook, that is, a connected skew shape with no $2\times 2$ square.
\end{itemize}
Let $\mathcal R_{\lambda,\nu}$ be the set of rim-hook tableaux of shape $\lambda$ and type $\nu$.
The height $\hgt(H)$ of a rim-hook $H$ is one less than the number of rows of $H$.  The height $\hgt(R)$ of a rim-hook tableau $R$ is the sum of the heights of the rim-hooks appearing in $R$. 
For each $\nu \in \Par(n)$, let  $$z_\nu:=\prod_i m_i(\nu)! \, i^{m_i(\nu)}$$ where $m_i(\nu)$ is the multiplicity of $i$ in $\nu$ for each $i $.

By the  Murnaghan-Nakayama rule and Theorem~\ref{schurcon}, for each $\nu \in \Par(n)$, the coefficient  $\tilde c_{\nu}(t)$ of $z_\nu^{-1} p_\nu$ in the $p$-basis expansion of $X_G(\x,t)$ is given by
$$\tilde c_{\nu}(t) = \sum_{\scriptsize\begin{array}{c} T \in \mathcal T_P  \\ R\in \mathcal R_{\lambda(T),\nu} \end{array}}(-1)^{\hgt(R)}t^{\inv_G(T)}.$$ 
Since  any rim-hook tableau of type $(n)$ must be a hook filled with $1$'s, we have 
\begin{equation} \label{eqhook} \tilde c_{(n)}(t) = (-1)^{n-1} \sum_{\scriptsize\begin{array}{c} T \in \mathcal T_P  \\ \lambda(T) \in \mathcal H_n \end{array}} (-1)^{r(T)-1} t^{\inv_G(T)},\end{equation}
where $\mathcal H_n$ is the set of hooks of size $n$ and $r(T)$ equals the length of the first row of $T$.

We will construct a sign-reversing, $\inv_G$-preserving involution $\varphi$ on the set of  hook-shaped tableaux in $\mathcal T_P$, which serves to reduce the sum in 
(\ref{eqhook}) to  the sum of $t^{\inv_G(T)}$ over  certain $P$-tableaux $T$ of shape 
$1^n$.   When one reads these tableaux from top down one gets a word in $\mathcal N_P$. 

The involution $\varphi$ is defined as follows:  
Let $T \in \mathcal T_P$ have hook shape.  The {\em arm} of $T$ is the top row of $T$ not including the leftmost entry and the {\em leg} of $T$ is the first column of $T$.    If $\lambda(T) \ne 1^n$,
let $a$ be the right most entry of the arm of $T$ and let $b$ be the lowest entry of the leg  of $T$ satisfying 
$b >_P  x$ for all $x$ higher than $b$ in the leg.   If $b >_P a$ then move $b$ to the right of $a$ in the arm to obtain $\varphi(T)$.   If $b \not >_P a$ then to obtain $\varphi(T)$ move $a$ to the leg by inserting it  as low as possible so that $a >_P y$ for all $y$ higher than  $a$ in the leg.  
Now if $\lambda(T) = 1^n$  then let $b$ be as above.  If $b$ is not at the top of the leg then move $b$ to the empty arm of $T$ to obtain $\varphi(T)$.  If  $b$ is at the top of the leg then let $\varphi(T)=T$, so that $T$ is a fixed point of the map $
\varphi$. 

Clearly $\varphi(T)$ is a hook-shaped $P$-tableaux and it is easy to check  that $\varphi^2(T) = T$ in all cases.   If $T$ is not a fixed point then $r(\varphi(T)) =  r(T) \pm 1$; 
hence $\varphi$ is a sign-reversing involution.  It is also easy to see that  $\inv_G(\varphi(T)) = \inv_G(T)$.  
 It follows that $$\tilde c_{(n)}(t) = (-1)^{n-1}  \sum_{T \in \mathcal F_P}  t^{\inv_G(T)},$$ where $\mathcal F_P$ is the set of fixed points of $\varphi$.  
 For $T \in \mathcal F_P$, let $\sigma(T)$ be the word obtained by reading $T$ from top down ($T$ is a 
 single column).   Clearly $\sigma$ is a bijection from $\mathcal F_P$ to  $\mathcal N_P$.  Since $\inv_G(T) =  \inv_G(\sigma(T))$, we conclude 
that $\tilde c_{(n)}(t) = (-1)^{n-1} \sum_{\sigma \in \mathcal N_P} t^{\inv_{G}(\sigma)}$.  Equation (\ref{pcoefeq2})  follows from this and the fact that 
$\omega p_n = (-1)^{n-1} p_n$.
\end{proof}

\begin{proof}[Proof of Lemma~\ref{pcoeflem3}] 

{\bf Part 1:} For $[k] \subseteq P$, let  $$\mathcal N_{P,k}^{=1} := \{\sigma \in \sg_k : \sigma(1) = 1, \,\,\, \sigma \mbox{ has neither $P$-descents}$$ $$ \mbox{ nor nontrivial left-to-right $P$-maxima}\}.$$ (Here we are viewing permutations in $\sg_k$ as words over alphabet $[k]$.)
We begin by  proving   that  for all $k \ge 1$ such that $[k] \subseteq P$,
\begin{equation} \label{eqprodk} \sum_{\sigma \in \mathcal N_{P,k}^{=1} } t^{\inv_{G}(\sigma)} = \prod_{i=2}^k [b_i]_t  \,.
\end{equation} 
This is equivalent to the recurrence 
\begin{equation}\label{b.recurrence} \sum_{\sigma \in \mathcal N_{P,k}^{=1}}t^{\inv_{G}(\sigma)} = [b_k]_t \sum_{\sigma \in \mathcal N_{P,k-1}^{=1}}t^{\inv_{G}(\sigma)} ,\end{equation} 
for all $k \ge 2$ with $[k] \subseteq P$.  To prove this recurrence we show that each word in $\mathcal N_{P,k}^{=1}$ can be obtained by inserting $k$ into a word in $\mathcal N_{P,k-1}^{=1}$ in one of $b_k$ allowable places.  

First we claim that removing $k$ from any word  $\sigma \in
\mathcal N_{P,k}^{=1}$   introduces neither a $P$-descent nor a nontrivial left-to-right $P$-maximal letter. 
Suppose a $P$-descent is introduced. 
Let  $s$ be the letter to the immediate left of $k$ in $\sigma$ and let $t$ be the letter to the immediate right of $k$.  Then 
$s >_P t$ and $k \not >_P t$.  It follows that $k$ is not comparable to $s$ or $t$.  Since $k > s,t$, this contradicts the fact that $P$ is a natural unit interval order.  We conclude that removal of $k$ from $\sigma$ does not introduce a $P$-descent.   Now suppose removal of $k$ from $\sigma$ introduces a nontrivial left-to-right $P$-maximal letter $m$.  Then $k$ is to the left of $m$ in $\sigma$, and  since $k$ is not a left-to-right maximal letter of $\sigma$, there is a letter $s$ to the left of $k$ such that $s \not <_P k$.  But $s <_P m$ since $s$ is left of $m$ and $m$ is a left-to-right maximal letter of $\sigma$ with $k$ removed.  Again since $k > m,s$, the assumption that $P$ is a natural unit interval order is contradicted.  Therefore removal of $k$ from $\sigma$ does not introduce a nontrivial left-to-right $P$-maximum and our   claim is proved.

Now let $B_k=  \{i \in [k-1]: \{i,k\} \in E(G)\}$; so that $|B_k| = b_k$.  Let $\alpha \in \mathcal N_{P,k-1}^{=1}$. The 
letters of $B_k$ form a subword $\alpha_{i_1},\alpha_{i_2},\dots, \alpha_{i_{b_k}}$ of $\alpha$.
 To avoid introducing a $P$-descent when inserting $k$ into  $\alpha$, the insertion must be made at the end of $\alpha$ or just to the left of  one of the $\alpha_{i_j}$'s.  To also avoid introducing a nontrivial left-to-right $P$-maximum or violating the condition that the first letter is 1,   $k$ must not be inserted just  to the left of  $\alpha_{i_1}$.  If $k$ is inserted at the 
end of $\alpha$ then no $G$-inversions are introduced. If  $k$ is inserted in $\alpha$ to  the immediate left of  
$\alpha_{i_j}$  then $b_k-j+1$ $G$-inversions are introduced.   We can therefore conclude that (\ref{b.recurrence}) holds.

{\bf Part 2:} The proof of (\ref{pcoefeq3}) is by induction on $n$.  Assume the result is true for posets of size $n-1$.  We may assume without loss of generality that $P$ is a natural unit interval order on $[n]$.
 It follows from Proposition~\ref{natunitprop2}  that 
\begin{equation} \label{eqa} \{ i \in \{2,\dots, n\} : \{1,i\} \in E(G) \}= \{2,3,\dots,m_1\}\end{equation}  and 
\begin{equation} \label {eqbj} \{1,i\} \in E(G) \implies b_i = i-1. \end{equation}

Let  $\mathcal N_P^{\ne 1} := \{\sigma \in \mathcal N_P : \sigma_1 \ne 1\}$,
let $P^1$ be the induced subposet of $P$  on $[n]\setminus \{1\}$, and let $G^1=\inc(P_1)$.  
We claim that 
\begin{equation} \label{ne1eq} \sum_{\sigma \in \mathcal N_P^{\ne 1}} t^{\inv_G(\sigma)} = t [m_1-1]_t \sum_{\sigma \in \mathcal N_{P^{1}}} t^{\inv_{G^1}(\sigma)}.
\end{equation}
To prove (\ref{ne1eq}) we  show that  each word in $\mathcal N_P^{\ne 1}$ can be obtained by inserting $1$ into a word in $\mathcal N_{P^{1}}$ in one of $m_1-1$ allowable places.
By an argument similar to the argument used to prove (\ref{b.recurrence}), we observe that removing $1$ from any word  $\sigma \in \mathcal N_P^{\ne 1} $ introduces neither a $P$-descent nor a nontrivial left-to-right $P$-maximum.  
Given a word $\alpha \in \mathcal N_{P^1}$, in order to avoid creating a $P$-descent or having $1$ as  the leftmost   letter when inserting $1$ into $\alpha$,  by (\ref{eqa}) the insertion must be made immediately to the right of a letter in $\{2,3,\dots,m_1\}$. Equation (\ref{ne1eq}) is a consequence of this.

Now by (\ref{ne1eq}), (\ref{eqa}), and the induction hypothesis  applied to the natural unit interval order $P^1$ we have
\begin{eqnarray} \nonumber \sum_{\sigma \in \mathcal N_P^{\ne 1}} t^{\inv_G(\sigma)} 
&=& t [m_1-1]_t \,\,[n-1]_t \prod_{i=3}^{m_1} [b_i-1]_t \prod_{i=m_1+1}^n [b_i]_t  \\
\nonumber &=& t [n-1]_t \,\, [m_1-1]_t  \prod_{i=3}^{m_1} [i-2]_t  \prod_{i=m_1+1}^n [b_i]_t 
\\ \nonumber &=& t [n-1]_t \,\, [m_1-1]_t! \prod_{i=m_1+1}^n [b_i]_t 
\\ \nonumber &=& t [n-1]_t \,\, \prod_{i=2}^{m_1} [b_i]_t \prod_{i=m_1+1}^n [b_i]_t 
\\ \label{eqproda1} &=& t [n-1]_t \,\, \prod_{i=2}^{n} [b_i]_t ,
\end{eqnarray}
with the second and fourth equations following from (\ref{eqbj}).

Now by  (\ref{eqprodk}) and (\ref{eqproda1}), we can  conclude that  \begin{eqnarray*} \sum_{\sigma \in \mathcal N_P}  t^{\inv_G(\sigma)} &=& \sum_{\sigma \in \mathcal N_{P,n}^{=1}} t^{\inv_G(\sigma) } +  \sum_{\sigma \in \mathcal N_P^{\ne 1}  }t^{\inv_G(\sigma)} \\
&=&( \prod_{i=2}^n [b_i]_t )+( t[n-1]_t \prod_{i=2}^n [b_i]_t )\\
&=& [n]_t \prod_{i=2}^n [b_i]_t,
 \end{eqnarray*}
 as desired.
 \end{proof}

  In Conjecture \ref{powercon1} and Proposition~\ref{powercon2} below, we present two proposed formulae for coefficients of the $p$-basis expansion of $\omega X_{G}(\x,t)$.
    For a partition $\mu=(\mu_1 \ge \mu_2 \ge \dots \ge \mu_l)$  of $n$ and  a natural unit interval order $P$ on $[n]$, 
   let $\mathcal N_{P,\mu}$ be the set of permutations $\sigma \in \sg_n$ such that when we break $\sigma$ (in one line 
   notation) into contiguous segments of lengths  $\mu_1,\ldots,\mu_l$ (from left to right), each contiguous segment 
   has neither a $P$-descent nor a nontrivial left-to-right $P$-maximum. Thus $\mathcal N_{P,(n)} = \mathcal N_P$.

 \begin{con}\footnote{See Section~\ref{newsec}: {\em Recent developments}}  \label{powercon1} Let $P$ be a natural unit interval order on $[n]$ with incomparability graph $G$.  Then
 \begin{equation} \label{pexpandeq}  \omega X_G(\x,t) = \sum_{\mu \vdash n}  z_\mu^{-1} p_\mu  \sum_{\sigma \in \mathcal N_{P,\mu}} t^{\inv_G(\sigma)}.\end{equation}  Moreover, the palindromic polynomial  $\sum_{\sigma \in \mathcal N_{P,\mu}} t^{\inv_G(\sigma)}$ is unimodal.
  \end{con}

Unimodality of $\sum_{\sigma \in \mathcal N_{P,\mu}} t^{\inv_G(\sigma)}$ is a consequence of $p$-unimodality of $\omega X_G(\x,t)$, which in turn  is a consequence of 
Conjecture~\ref{quasistan}. It can be shown that the formula for the coefficient of $p_\mu$ predicted by Conjecture~\ref{powercon1} reduces to the formula given \cite[Theorem 2.6]{St3} when $t=1$.
In the case that $\mu = (n)$ the formula clearly reduces to (\ref{pcoefeq2}). 

We now give an alternative conjectural characterization of the coefficients of $z_\mu^{-1} p_\mu$.  Let 
$P$ be a natural unit interval order of size $n$.  Given a word $a_1 \cdots  a_k$ where each $a_i \in P$, we say that $i$ is a $P$-ascent if $a_i <_P a_{i+1}$.  Let $\widetilde {\mathcal N}_{P}$ be the set of 
words $\alpha\in P^n$ that have distinct letters, start with the smallest letter of $P$ (in the natural order on $\pp$),  and have no $P$-ascents.  

\begin{lemma} \label{lemalt} Let 
$P$ be a natural unit interval order of size $n$ and let $G=\inc(P)$.  Then
$$\sum_{\sigma \in \mathcal N_P} t^{\inv_G(\sigma)} = [n]_t \sum_{\sigma \in \widetilde {\mathcal N}_{P}} t^{\inv_G(\sigma)}$$
\end{lemma} 

\begin{proof} The result   follows from (\ref{pcoefeq3})  and the  equation,
 $$\sum_{\sigma \in \widetilde {\mathcal N}_{P}} t^{\inv_G(\sigma)} = \prod_{i=2}^n [b_i]_t,$$
 where $b_i$ is as in (\ref{bi.eq}).  The proof of this equation is similar to that of (\ref{eqprodk}) and is left to the reader.
\end{proof}

Now let $P$ be a natural unit interval order on $[n]$.  For $\mu=(\mu_1 \ge \mu_2 \ge \dots \ge \mu_l) \in \Par(n)$, 
   let $\widetilde {\mathcal N}_{P,\mu}$ be the set of permutations $\sigma \in \sg_n$ such that when we break $\sigma$ (in one line 
   notation) into contiguous segments of lengths  $\mu_1,\ldots,\mu_l$ (from left to right), each contiguous segment 
   has no $P$-ascent and begins with the  smallest letter of the segment in the natural order on $[n]$. 

  \begin{prop}   \label{powercon2} Let $P$ be a natural unit interval order on $[n]$ with incomparability graph $G$.  Then for all $\mu\in \Par( n)$,
\begin{equation} \label{powersumeq} \sum_{\sigma \in \mathcal N_{P,\mu}} t^{\inv_G(\sigma)} = \prod_{i=1}^{l(\mu)} [\mu_i]_t\sum_{\sigma \in \widetilde{\mathcal N}_{P,\mu}} t^{\inv_G(\sigma)}.\end{equation}  
Consequently,  (\ref{pexpandeq}) is equivalent to 
\begin{equation} \label{powereq} \omega X_G(\x,t) =  \sum_{\mu \vdash n}  z_\mu^{-1} p_\mu \prod_{i=1}^{l(\mu)} [\mu_i]_t\sum_{\sigma \in \widetilde{\mathcal N}_{P,\mu}} t^{\inv_G(\sigma)}.\end{equation}
 \end{prop}
 
 \begin{proof} A set composition of $[n]$ is a sequence ${\bf B}=(B_1,\dots,B_l)$ of mutually disjoint nonempty sets  whose union in $[n]$.  The type of ${\bf B}$ is the sequence $(|B_1|,\dots,|B_l|)$. Given a set $S \subseteq \pp$, let $ S\!\uparrow$ be the word obtained by arranging the elements of $S$ in increasing order and given a set composition ${\bf B}=(B_1,\dots,B_l)$ on $[n]$, let ${\bf B}\!\uparrow$ be the permutation $B_1\!\uparrow \circ \dots \circ B_l\!\uparrow$, where  $\circ $ denotes concatenation.
 
Now let $\mu$ be a partition of $n$ of length $l$ and let    $\sigma \in \S_n$.  For each  $i \in [l]$, let $\sigma^{(i)}$ be the $i$th segment of $\sigma$ under the segmentation determined by $\mu$ and let $B_i(\sigma)$ be the set of letters in $\sigma^{(i)}$.  So ${\bf B}(\sigma):=(B_1(\sigma), \dots, B_l(\sigma)) \in C_\mu$, where $C_\mu$ is the set of set compositions of $[n]$ of type $\mu$.  
We have $$\inv_G(\sigma) = \inv_G({\bf B}(\sigma)\!\uparrow ) + \sum_{i=1}^l \inv_G(\sigma^{(i)}).$$  It follows that 
 $$ \sum_{\sigma \in \mathcal N_{P,\mu}} t^{\inv_G(\sigma)} = \sum_{{\bf B} \in C_\mu} t^{\inv_G({\bf B}\uparrow)} \prod_{i=1}^l \sum_{\alpha \in \mathcal N_{P|B_i}} t^{\inv_G(\alpha)}$$ 
 and  $$ \sum_{\sigma \in \widetilde{ \mathcal N}_{P,\mu}} t^{\inv_G(\sigma)} = \sum_{{\bf B} \in C_\mu} t^{\inv_G({\bf B}\uparrow)} \prod_{i=1}^l \sum_{\alpha \in \widetilde{\mathcal N}_{P|B_i}} t^{\inv_G(\alpha)},$$ 
 where $P|B_i$ is the restriction of $P$ to the set $B_i$.  The result now follows from Lemma~\ref{lemalt}.
  \end{proof}

\begin{prop} Conjecture~\ref{powercon1} is true for $P:=P_{n,2}$. \end{prop} 

\begin{proof} 
Let $\sigma \in \widetilde{\mathcal N}_{P,\mu}$ and set $s_i = 1+\sum_{j=1}^{i-1} \mu_j$ and $t_i = \sum_{j=1}^{i} \mu_j$ for all $i \in [l(\mu)]$.  Then the $i$th contiguous segment of $\sigma$ is the sequence $\sigma(s_i),\sigma(s_i)+1,\dots,\sigma(s_i)+\mu_i-1= \sigma(t_i)$ since the segment has no $P$-ascents and $\sigma(s_i)$ is the smallest letter 
in the segment. There are no $G_{n,2}$-inversions within these segments.  The  $G_{n,2}$-inversions of $\sigma$ 
are of the form $(\sigma(s_{i}),\sigma(t_j))$, where $i <j$ and $\sigma(s_i) = \sigma(t_j)+1$.  Now let  $\overline{\sigma}$ be the 
permutation in $\sg_{l(\mu)}$ obtained from $ \sigma(s_1),\dots,\sigma(s_{l(\mu)})$ by replacing the $k$th smallest letter 
by $k$ for each $k$.   Clearly $\sigma(s_i) = \sigma(t_j)+1$ if and only if $\overline{\sigma}(i) =  \overline{\sigma}(j)+1$.  
Hence $\inv_{G_{n,2}}(\sigma) = \inv_{G_{l(\mu),2}}(\overline{\sigma})$.  

Now
$$\sum_{\sigma \in \widetilde{\mathcal N}_{P,\mu}} t^{\inv_{G_{n,2}}(\sigma)} 
= \sum_{\sigma \in \sg_{l(\mu)}} t^{\inv_{G_{l(\mu),2}}(\sigma)},$$
as the map
$\varphi: \widetilde{\mathcal N}_{P,\mu} \to \sg_{l(\mu)}$  defined by  $\varphi(\sigma) = \overline{\sigma}$ is  a bijection.

Since $\inv_{G_{n,2}}(\sigma) = \des(\sigma^{-1})$ for all $\sigma \in \sg_n$, by the definition of Eulerian polynomials given in (\ref{defEuler}), we now have
$$\sum_{\sigma \in \widetilde{\mathcal N}_{P,\mu}} t^{\inv_{G_{n,2}}(\sigma)} = A_{l(\mu)}(t).$$
Hence for $P:=P_{n,2}$ and $G:=G_{n,2}$, equation (\ref{powereq}) reduces to 
\begin{equation} \label{stemeq} \omega X_{G_{n,2}}(\x,t) = \sum_{\mu \vdash n}  z_\mu^{-1} p_\mu  
A_{l(\mu)}(t) \prod_{i=1}^{l(\mu)} [\mu_i]_t .\end{equation}
Stembridge \cite{Ste1} showed that the coefficient of $z^n$ on the right hand side of (\ref{pathchrom}) is equal to $\omega$ applied to the right hand side of (\ref{stemeq}).  Hence (\ref{stemeq}) holds and the conjecture is verified for $P_{n,2}$.
\end{proof}

\section{Some calculations} \label{calcsec}
 Recall the definition of $P(\bf m)$ given prior to Proposition~\ref{natunitprop2}.
 
  \begin{prop} \label{m3n}
Let $P=P({\bf m})$, where ${\bf m}=(m_1,\dots,m_{n-1})$ is a weakly increasing sequence of positive integers such that  $i \le m_i$ for each $i$. 
 Assume that $m_1 \geq 2$ and that $m_i=n$ for $i=3,\dots,n-1$. 
Define $B_\lambda(t)$ and $C_\lambda(t)$ for each partition $\lambda$ of $n$ by $$X_{\inc(P)}(\xx,t)=\sum_{\lambda\in \Par(n)}B_\lambda(t)s_\lambda=\sum_{\lambda \in \Par(n)} C_\lambda(t)e_\lambda.$$
Define \begin{eqnarray*} 
D(t)&:=& t^{m_1-1} [n-m_1]_t [m_2 -2]_t + t^{m_2-2}  [n-m_2]_t [m_1]_t  \\
 \mbox{ and } 
\\ E(t) &:=& t^{m_1-1}[n-3]_t[n-m_1]_t[m_2-2]_t+t^{m_2-2}[n-1]_t[n-m_2]_t[m_1-2]_t.\end{eqnarray*}
The following equalities hold:
\begin{enumerate}
\item if $\lambda\not\in \{1^n,21^{n-2},2^21^{n-4}\}$, then $B_\lambda(t)=0$, \item $B_{1^n}(t)=[m_1]_t[m_2-1]_t[n-2]_t!$, \item $B_{21^{n-2}}(t)=[n-3]_t!D(t)$, 
 \item $C_{(n-2,2)}(t)=B_{2^21^{n-4}}(t)=[n-4]_t![2]_t[n-m_2]_t [n-m_1-1]_t  t^{m_1+m_2-4}$, \item if $\lambda \not\in \{(n),(n-1,1),(n-2,2)\}$, then $C_\lambda(t)=0$, \item $C_{(n)}(t)=[n]_t [n-3]_t![m_1-1]_t [m_2-2]_t$,
\item $C_{(n-1,1)}(t)=[n-4]_t!E(t)$.
\end{enumerate}
\end{prop}

\begin{proof} Let $G=\inc(P)$.
Note first that (2) follows from Theorem~\ref{s1n} and (6) follows from Corollary~\ref{ecoefcor}.  Next, if $x<_Py$ then $x \in \{1,2\}$.  Therefore, $P$ has no chain of length two, since $\{1,2\} \in E(G)$.  Moreover, $P$ does not contain three pairwise disjoint chains of length one.  Now (1) follows from Theorem \ref{schurcon}.  

 Using  the    equalities \begin{eqnarray*} e_n & = & s_{1^n}, \label{es1} \\ e_{(n-1,1)} & = & s_{1^n}+s_{21^{n-2}} \label{es2}, \\ e_{(n-2,2)} & = & s_{1^n}+s_{21^{n-2}}+s_{2^21^{n-4}}, \label{es3} \end{eqnarray*} which  are obtained from the Pieri rule, we see that (5) follows from (1), that $$C_{(n-2,2)}(t)=B_{2^21^{n-4}}(t),$$ and that \begin{equation} \label{cn11} C_{(n-1,1)}(t)=B_{21^{n-2}}(t)-B_{2^21^{n-4}}(t). \end{equation}
It remains to prove (3), the second equality of (4), and \begin{equation} \label{diff} [n-3]_tD(t)=[2]_t[n-m_2]_t[n-m_1-1]_t t^{m_1+m_2-4}+E(t).  \end{equation} Indeed, (7) follows from (\ref{cn11}), (3), (4) and (\ref{diff}).

We begin with (3).  For $i \in \{1,2\}$ and $m_i<j \leq n$, let ${\mathcal T}_{P,i,j}$ be the set of all $P$-tableaux $T$ of shape $21^{n-2}$ such that the first row of $T$ is occupied by $i<_P j$.  Set $$Y_{i,j}(t):=\sum_{T \in {\mathcal T}_{P,i,j}}t^{\inv_{G}(T)}.$$  Each $P$-tableau of shape $21^{n-2}$ lies in exactly one $\mathcal T_{P,i,j}$.  By Theorem~\ref{schurcon}, \begin{equation} \label{byt} B_{21^{n-2}}(t)=\sum_{i<_P j}Y_{i,j}(t). \end{equation}  If $T \in {\mathcal T}_{P,i,j}$ and $P^\prime$ is the subposet of $P$ induced on $P \setminus \{i,j\}$, then upon removing the first row of $T$ we obtain a $P^\prime$-tableau of shape $1^{n-2}$.  Let $G^\prime = \inc(P^\prime)$ and note that $P^\prime$ is a natural unit interval order.
We have $$ Y_{i,j}(t) = t^a \sum_{\scriptsize \begin{array}{c} T^\prime \in \mathcal T_{P^\prime}\\ \lambda(T^\prime)= 1^{n-2}\end{array}} t^{\inv_{G^\prime}(T^\prime)},$$ where \begin{itemize}
\item[(1)] $a = {j-2}$ if either $i=1$ and $m_1< j \le m_2$ or $i=2$ and $m_2 < j \le n$, and \item[(2)] $a={j-3}$ if $i=1$ and $m_2 < j \le n$.  \end{itemize} Hence by Theorem~\ref{s1n},
\begin{equation} \label{yeq} Y_{ij}(t)=\begin{cases} t^{j-2}[m_2-2]_t[n-3]_t! &\mbox{if } i=1,m_1<j \leq m_2, \\ t^{j-3}[m_2-1]_t[n-3]_t! &\mbox{if } i=1,m_2<j \leq n, \\ t^{j-2}[m_1-1]_t[n-3]_t! &\mbox{if } i=2,m_2<j \leq n. \end{cases} \end{equation}

Now by (\ref{byt}), (\ref{yeq}) and the fact that $$\sum_{b<j \leq c}t^j=t^{b+1}[c-b]_t$$ for arbitrary positive integers $b,c$, we have,
\vspace{.1in}\newline$\displaystyle{\frac{B_{21^{n-2}}(t)}{[n-3]_t!} =}$ $$ t^{m_1-1}[m_2-m_1]_t[m_2-2]_t +t^{m_2-2}[n-m_2]_t[m_2-1]_t+t^{m_2-1}[n-m_2]_t[m_1-1]_t.$$
The first term of the right side of the above equation equals 
$$t^{m_1-1}([n-m_1]_t-t^{m_2-m_1}[n-m_2]_t)[m_2-2]_t,$$ 
the second term equals
$$t^{m_2-2}[n-m_2]_t (1+t[m_2-2]_t),$$
and the third term equals
$$t^{m_2-2}[n-m_2]_t([m_1]_t-1).$$
Adding these expressions together yields (3).

Arguments similar to those used to prove (3) prove (4).  For $m_1<i \leq n$ and $m_2<j \leq n$, now let ${\mathcal T}_{P,i,j}$ be the set of all $P$-tableaux $T$ of shape $2^21^{n-4}$ such that the first row of $T$ is occupied by $1$ and $i$ and the second row is occupied by $2$ and $j$, and let ${\mathcal T}^*_{P,j,i}$ be the set of all $P$-tableaux $T$ of shape $2^21^{n-4}$ such that the first row of $T$ is occupied by $2$ and $j$ and the  second row is occupied by $1$ and $i$.  Set $$Y_{i,j}(t):=\sum_{T \in {\mathcal T}_{P,i,j}}t^{\inv_{G}(T)}\qquad \mbox{and}\qquad Y^*_{j,i}(t):=\sum_{T \in {\mathcal T}^*_{P,j,i}}t^{\inv_{G}(T)}.$$  Then \begin{equation} \label{byijt} B_{2^21^{n-4}}(t)=\sum_{i,j}Y_{i,j}(t)+Y^*_{j,i}(t). \end{equation}  Having fixed the first two rows of a $P$-tabelau $T$ of shape $2^21^{n-4}$, we can place the remaining $n-4$ elements of $[n]$ into $T$ in arbitrary order, as there are no relations in $P$ among these elements.  It follows that \begin{equation} \label{yij} Y_{i,j}(t)=\left\{\begin{array}{ll} t^{i+j-6}[n-4]_t! & \mbox{if } m_1<i \leq m_2, \\ t^{i+j-6}[n-4]_t! & \mbox{if } m_2 <j<i \leq n, \\ t^{i+j-7}[n-4]_t! & \mbox{if }  m_2<i<j \leq n, \end{array} \right. \end{equation} and \begin{equation} \label{yji} Y^*_{j,i}(t)=\left\{\begin{array}{ll} t^{i+j-5}[n-4]_t! & \mbox{if } m_1<i \leq m_2, \\ t^{i+j-5}[n-4]_t! & \mbox{if } m_2<i<j\leq n, \\ t^{i+j-6}[n-4]_t! & \mbox{if } m_2<j<i \leq n. \end{array} \right.  \end{equation}  It is striaghtforward to show by induction that for nonnegative integers $a\le b$, \begin{equation} \label{absum} \sum_{a<i<j \leq b} t^{i+j}=t^{2a+3}\left[\begin{array}{c} b-a \\ 2\end{array}\right]_t.  \end{equation} Combining (\ref{byijt}), (\ref{yij}), (\ref{yji}), and (\ref{absum}), we get 
\begin{eqnarray*} \label{yijsum} \frac{B_{2^21^{n-4}}(t)}{[n-4]_t!} & = & (1+t)\sum_{i=m_1+1}^{m_2}t^{i-3}\sum_{j=m_2+1}^{n}t^{j-3} + (1+t)^2\sum_{m_2<k<l \leq n}t^{k+l-7} \\ & = & [2]_t t^{m_1+m_2-4}[m_2-m_1]_t[n-m_2]_t+[2]_t^2t^{2m_2-4} \left[\begin{array}{c} n-m_2\\ {2}\end{array}\right]_t
\\ &=&   t^{m_1+m_2-4} [2]_t [n-m_2]_t ([m_2 -m_1]_t +[n-m_2-1]_t t^{m_2-m_1} ).\end{eqnarray*}  
 Claim (4) follows quickly.
 
 Finally, (\ref{diff}) can be proved by a direct calculation, which we present here.  
 By simple manipulation we see that the right hand side of (\ref{diff}) is given by
\begin{eqnarray*}\mbox {RHS} &=& t^{m_2-2} [n-m_2]_t ( t^{m_1-2} [2]_t [n-m_1 -1]_t+ [n-1]_t[m_1-2]_t) \\ 
&\phantom{=}&   + t^{m_1-1} [n-3]_t [n-m_1]_t [m_2 -2]_t .
\end{eqnarray*}
Using the identity $t^b [a-b]_t = [a]_t -[b]_t$, we obtain
$$t^{m_1-2} [2]_t [n-m_1 -1]_t+ [n-1]_t[m_1-2]_t$$
\begin{eqnarray*}  &=& [2]_t([n-3]_t - [m_1-2]_t) + [n-1]_t [m_1-2]_t
\\ &=& [2]_t [n-3]_t +  t^2[n-3]_t [m_1-2]_t
\\ &=& [n-3]_t [m_1]_t
.\end{eqnarray*}
Hence, 
$$ \mbox {RHS} =  [n-3]_t (t^{m_2-2}  [n-m_2]_t [m_1]_t + t^{m_1-1} [n-m_1]_t [m_2 -2]_t) $$
as desired.
\end{proof}

The next three corollaries follow easily from Proposition \ref{m3n}.

\begin{cor} \label{epos3n}
Let $P=P(m_1,\ldots,m_{n-1})$ be a natural unit interval order.  If $m_3=n$ then $X_{\inc(P)}(\xx,t)$ is $e$-positive and $e$-unimodal.
\end{cor}

\begin{proof}Apply Propositions~\ref{tooluni} and~\ref{tooluni2}.
\end{proof}

\begin{cor} For $n \ge 2$,\begin{eqnarray*} X_{G_{n,n-1}}(\xx,t)&=& [n-2]_t![n-1]_t^2 s_{1^n} + [n-2]_t! t^{n-2} s_{21^{n-2}}
\\&=&[n-2]_t![n]_t[n-2]_t e_{n} + [n-2]_t! t^{n-2} e_{(n-1,1)}.
\end{eqnarray*}
\end{cor}

\begin{cor} \label{pnn2}
For $n \geq 3$, \begin{eqnarray*} X_{G_{n,n-2}}(\xx,t) & = & [n-2]_t^2[n-2]_t!s_{1^n} 
\\ & + &  t^{n-3} [n-3]_t! ([2]_t[n-3]_t +[n-2]_t)s_{21^{n-2}} 
\\ & + & t^{2n-7}[2]_t[n-4]_t!s_{2^21^{n-4}}\\ & = & [n]_t[n-3]_t^2[n-3]_t!e_{n} \\ & + & t^{n-3}[n-4]_t!([2]_t[n-3]_t^2+[n-1]_t[n-4]_t)e_{(n-1,1)} \\ & + & t^{2n-7}[2]_t[n-4]_t!e_{(n-2,2)}. \end{eqnarray*}
\end{cor}

\section{Specialization}  \label{specsec}

We review two important ways to specialize a quasisymmetric function. Let
$R[[q]]$ denote the ring of formal power series in variable $q$ with coefficients in $R$.  
For a quasisymmetric function $Q(x_1,x_2,\ldots) \in \cq_R$, the {\it stable principal specialization} $\Ps(Q) \in R[[q]]$ is, by definition, obtained from $Q$ by substituting $q^{i-1}$ for $x_i$ for each $i \in \pp$. That is,
$$\Ps(Q) := Q(1,q,q^2,q^3,\dots).$$  For each $m \in \pp$, the {\it principal specialization} $\Ps_m(Q)$ is given by
$$\Ps_m(Q) := Q(1,q,\dots, q^{m-1},0,0,\dots).$$

\begin{lemma}[{\cite[Lemma 5.2]{gr}}] \label{desspec}For all $n \ge 1$ and $S \subseteq [n-1]$, we have
\begin{equation}\label{spec1}\Ps(F_{S,n})= \frac {q^{\sum_{i\in S} i}} {(q;q)_n} \end{equation} and
 \begin{equation} \label{spec2}\sum_{m \ge 1} \Ps_m( F_{S,n}) p^m = {p^{|S|+1} q^{\sum_{i\in S} i} \over (p;q)_{n+1}}, 
 \end{equation}
 where $$(p;q)_n = \prod_{j=1}^n (1-pq^{j-1}).$$ \end{lemma}

It follows from Lemma~\ref{desspec} that
\begin{equation} \label{stabh}  \Ps (h_n) = \Ps(F_{n,\emptyset}) = (q,q)^{-1}_n\end{equation} for all $n$ and therefore if $(k_1,\dots,k_m)$ is a  composition of $n$,
\begin{equation}\label{spech}(q,q)_n\, \Ps (h_{(k_1,\dots,k_m)} )= \left [\begin{array}{c} {n}\\{k_1,\dots,k_m}\end{array} \right ]_q : = {[n]_q!
\over [k_1]_q!\cdots [k_m]_q!} .\end{equation}

The following is an immediate consequence of Lemma~\ref{desspec}.  
\begin{prop} \label{unispecprop} Let $F$ denote the basis of fundamental quaisymmetric functions and suppose $f(\x,t) \in \cq_\Q^n[t]$ is 
$F$-positive  (resp. $F$-unimodal). Then $(q;q)_n \Ps (f(\x,t)) $ is $q$-positive (resp. $q$-unimodal) and $(p;q)_{n+1} \sum_{m \ge 1} \Ps_m( f(\x,t)) p^m $ is $(p,q)$-positive (resp. $(p,q)$-unimodal). 
\end{prop}

Since $b$-positivity implies $F$-positivity when $b$ is the Schur basis, the $e$ basis or the $h$ basis, Proposition~\ref{unispecprop}  holds for $f(\x,t) \in \Lambda_\Q^n[t]$ when $F$ is replaced by any of these bases.

\subsection{Generalized $q$-Eulerian polynomials} By applying Lemma~\ref{desspec} to the expansion of $\omega X_G(\x,t)$ in the basis of fundamental quasisymmetric functions given in Theorem~\ref{qchow}, we obtain the following result.  For  the incomparability graph $G$ of a  poset $P$ on $[n]$, define $$A_G(q,p,t):= \sum_{\sigma \in \S_n} t^{\inv_G(\sigma)} q^{\maj_P(\sigma)} p^{\des_P(\sigma)},$$
where  $\maj_P(\sigma) := \sum_{i \in \Des_P(\sigma)}i$ and $\des_P(\sigma) := |\Des_P(\sigma)|$.

\begin{thm}\label{speccor}  Let  $G$ be the incomparability graph of a  poset $P$ on $[n]$.   Then
\begin{equation}\label{stabprinceq} (q;q)_n \,\Ps(\omega X_G(\x,t) )= A_G(q,1,t)\end{equation}
and
\begin{equation}\label{princeq}(p;q)_{n+1} \, \sum_{m \ge 1} \Ps_m( \omega X_G(\x,t)) p^{m-1} = A_G(q,p,t).\end{equation}
\end{thm}

\begin{cor} \label{sympalin} Let  $G$ be the incomparability graph of a  poset  $P$ on $[n]$.  Then $A_G(1,p,t) $ is palindromic as a polynommial in $t$.
If $X_G(\x,t)$ is symmetric in $\x$ then $A_G(q,p,t) $ is palindromic as a polynomial in $t$.  
\end{cor}

\begin{proof} Recall the involution $\rho$ given in Section~\ref{basicsec}.  Note that for all $S\subseteq [n-1]$,  $\rho F_{n,S} = F_{n,n-S}$, where  $n-S:=\{n-i: i \in S\}$.  Hence 
$$(p;q)_{n+1} \, \sum_{m \ge 1} \Ps_m( \rho F_{n,S}) p^{m-1} = (pq^n)^{|S|} q^{-\sum_{i\in S} i}.$$Since $\omega$ and $\rho$ commute it follows from Proposition~\ref{rhoprop} that $$\rho(\omega X_G(\x,t)) = t^{|E|}\omega X_G(\x,t^{-1}).$$ It therefore follows from Theorem~\ref{qchow} that
$$ \sum_{\sigma \in \S_n} t^{\inv_G(\sigma)} \rho F_{n,\Des_{P}(\sigma)} = t^{|E|} \sum_{\sigma \in \S_n} t^{-\inv_G(\sigma)} F_{n,\Des_{P}(\sigma)}.$$  Taking principal specializations $\Ps_m$ of both sides yields
$$\sum_{\sigma \in \S_n} t^{\inv_G(\sigma)} (pq^n)^{\des_P(\sigma)}q^{-\maj_P(\sigma)} = t^{|E|} \sum_{\sigma \in \S_n} t^{-\inv_G(\sigma)} p^{\des_P(\sigma)} q^{\maj_P(\sigma)}.$$
Setting $q=1$ gives
$$\sum_{\sigma \in \S_n} t^{\inv_G(\sigma)} p^{\des_P(\sigma)} = t^{|E|} \sum_{\sigma \in \S_n} t^{-\inv_G(\sigma)} p^{\des_P(\sigma)} ,$$  which establishes palindromicity of $A_G(1,p,t)$.  

The palindromicity of $A_G(q,p,t)$ asserted in the second statement is a consequence of the palindromicity of $X_G(\x,t)$ obtained in Corollary~\ref{sympalincor}. \end{proof}

\begin{example}For the graph $G$ of Example~\ref{XGex}(a), $X_G(\x,t)$ is not symmetric in $\x$ and 
$$A_G(q,p,t)= (1+qp) + 2t + (1+q^2p) t^2 $$
is not palindromic in $t$, but $A_G(1,p,t)$ is palindromic.
 For the graph $G$ of Example~\ref{XGex}(b), $X_G(\x,t)$ is  symmetric in $\x$ and 
 $$A_G(q,p,t)   = 1 + (2+qp+q^2p)t + t^2$$ is palindromic in $t$.
\end{example}

By specializing Conjecture~\ref{quasistan} (or the weaker Schur-unimodality conjecture) and applying Proposition~\ref{unispecprop} we obtain the following purely combinatorial conjecture. This conjecture is a $(q,p)$-analog of a result of De Mari, Procesi, and Shayman \cite{DePrSh} discussed in Section~\ref{hesssec}.

\begin{con} \footnote{See Section~\ref{newsec}: {\em Recent developments}} \label{GEulcon} Let $G$ be the incomparability graph of a natural unit interval order on $[n]$.   Then $A_{G}(q,p,t)$  is a $(q,p)$-unimodal polynomial in $t$.
\end{con}

\vspace{.in} Now for each $r \in [n]$, let 
\begin{eqnarray*} \inv_{<r}(\sigma) &:=& |\{(i,j) : i < j, \,\,0<\sigma(i)-\sigma(j) < r\}|\\ 
\maj_{\ge r}(\sigma) &:=& \sum_{i : \sigma(i) - \sigma(i+1) \ge r}  i
\\ A^{(r)}_n(q,t)&:=&\sum_{\sigma \in \sg_n} q^{\maj_{\ge r}(\sigma)} t^{\inv_{<r}(\sigma)}.\end{eqnarray*}
Then $\inv_{<r} = \inv_{G_{n,r}}$, $\maj_{\ge r} = \maj_{P_{n,r}}$ and $A^{(r)}_n(q,t) = A_{G_{n,r}}(q,1,t)$.
The permutation statistic $\inv_{<r}+ \maj_{\ge r}$ was introduced by Rawlings  \cite{Ra} who proved that 
it is Mahonian for all $r$, that is
$$\sum_{\sigma \in \sg_n} q^{\maj_{\ge r}(\sigma)+ \inv_{<r}(\sigma)} = [n]_q !.$$
Note that this Mahonian statistic interpolates between $\maj$ (when $r=1$) and $\inv$  (when $r=n$).
It follows from a more general result of Kasraoui  \cite[Theorem 1.8]{Ka} that $\maj_P+\inv_{\inc(P)}$ is Mahonian for every natural unit interval order $P$.  Note that the Mahonian property fails for general $P$; see Example~\ref{XGex} (a).

By Theorem~\ref{speccor} we have

$$A^{(r)}_n(q,t) = (q;q)_n \Ps(\omega X_{G_{n,r}}(\x,t)).
$$
Hence by specializing the formulae in Table 1 of Section~\ref{esec} and using (\ref{spech}) we obtain Table~2 below. From this table and Propositions~\ref{tooluni} and~\ref{tooluni2}, we see that $A_n^{(r)}(q,t)$ is a $q$-unimodal polynomial in $t$ for $r\le 2$ and $r \ge n-2$.

 \begin{table}[htdp] 
\hspace*{-.2in}
\begin{tabular}{|@{}c@{}||l|l|l|l|l|} 
\hline{\color{blue} \scriptsize{$r$}} & {\color{blue}\scriptsize{$A^{(r)}_n(q,t)$} } \\
\hline
\hline{\color{blue} \scriptsize{1}} & \scriptsize{{$[n]_q!$}}  \\
\hline{\color{blue} \scriptsize{2}} & \scriptsize{$\displaystyle{{\sum_{m=1}^{\lfloor {n+1 \over 2} \rfloor}\,\,\sum_{{
k_1,\dots, k_m \ge 2 }}\,\,\,{\left[\begin{array}{c} n \\ k_1-1,k_2, \dots, k_m\end{array}\right]_q}\,\, {\color{red}t^{m-1} \!\!  \prod_{i=1}^m  [k_i-1]_{t} }}}$} \\
\hline {\color{blue} \vdots} & \\ 
\hline{\color{blue} \scriptsize{\,\,$r$\,\,}} & \scriptsize{${\color{red} [n]_t  [ r-1]_t^{n-r} [r-1]_t!} \qquad+\qquad$ ???}\\
\hline {\color{blue} \vdots} &\\ 
\hline{\color{blue} \scriptsize{\,\,$n-2$\,\,}} & \scriptsize{${\color{red} [n]_t  [ n-3]_t^2 [n-3]_t!}  +  [n]_q \,\,{\color{red}   t^{n-3}[n-4]_t!([2]_t[n-3]_t^2+[n-1]_t[n-4]_t)} $  } \\ & \scriptsize{$  + \left[\begin{array}{c} n \\ n-2,2\end{array}\right]_q{\color{red} t^{2n-7}[2]_t[n-4]_t!  } $} \\
\hline{\color{blue} \scriptsize{\,\,$n-1$\,\,}} &\scriptsize{$ {\color{red}[n]_t [n-2]_t [n-2]_t !  } + [n]_q {\color{red} t^{n-2} [n-2]_t! }$}\\
\hline{\color{blue} \scriptsize{$n$}} & \scriptsize{${\color{red} [n]_t! } $}
\\
\hline \end{tabular}
 \caption{Specialization of Table 1.}
\end{table}

Since $\inv_{<2}(\sigma) = \des(\sigma^{-1})$, it follows from the definition of Eulerian polynomials given in (\ref{defEuler}) that $A_n^{(2)}(1,t) = A_n(t)$.  We have the following $q$-analog of this observation, which involves the $q$-Eulerian polynomials introduced in \cite{ShWa1,ShWa}.

\begin{thm} \label{RawlingsqEulerTh} Let $A_n(q,t)$ be the $q$-Eulerian polynomial  
$$A_n(q,t) = \sum_{\sigma \in \sg_n} q^{\maj(\sigma) -\exc(\sigma)} t^{\exc(\sigma)}.$$ Then for all $n \ge 1$, $$A_n^{(2)}(q,t) = A_n(q,t) .$$
\end{thm}

\begin{proof}  Since $G_{n,2}$ is the path $G_n$, applying the involution  $\omega$ to both sides of (\ref{pathchrom}) yields
\begin{equation} \label{symEulereq} 1+\sum_{n\ge 1} \omega X_{G_{n,2}}(\x,t) = \frac{(1-t)H(z)}{H(zt) -t H(z)},\end{equation}
where $H(z) = \sum_{n \ge 0} h_n(\x) z^n$.  Now by  taking the stable principal specialization of both sides and applying~(\ref{stabprinceq}) and~(\ref{stabh}), we have
$$
1+\sum_{n \geq 1}A^{(2)}_n(q,t)\frac{z^n}{[n]_q!}=\frac{(1-t)\exp_q(z)}{\exp_q(tz)-t\exp_q(z)},
$$ where  $$\exp_q(z) := \sum_{n \ge 0} \frac {z^n}{[n]_q!}.$$
The result now follows from (\ref{expgeneq}).
\end{proof}

 \begin{prob} Find a direct bijective proof of Theorem~\ref{RawlingsqEulerTh}.  The classical bijection for the $q=1$ case does not work for the $q$-analog.  Also a  combinatorial description of the coefficients of the fundamental quasisymmetric functions in the  expansion of $X_{G_{n,2}}(\x,t)$, which is different from the one given here,  is given in \cite{ShWa} .  A combinatorial explanation of why these coefficients are equal would be interesting.
 \end{prob}
 
 \subsection{Evaluation of generalized $q$-Eulerian polynomials at roots of unity}
It is surprising that for every $n$th root of unity  $\xi_n$, the coefficients of  
$A^{2}_n(\xi_n,t)=A_n(\xi_n,t)$ are positive integers. 
 This is a consequence of the following result.

\begin{thm}[Sagan, Shareshian, and Wachs {\cite[Corollary 6.2]{SaShWa}}]  \label{unity}\hspace{.1in} Let $dm=n$ and let $\xi_d$ be any primitive $d$th root of unity.   Then
$$A_n(\xi_d,t) = [d]^m_t A_m(t) .$$ Consequently, $A_n(\xi_d,t)$ is a positive, palindromic, unimodal polynomial in $\Z[t]$.  \end{thm}

A key tool in the proof of Theorem~\ref{unity} is the following result, which is  implicit in \cite{De} and  stated explicitly in \cite{SaShWa}.

\begin{lemma}[see \cite{SaShWa}] \label{thdes} Suppose $u(q) \in \Z[q]$ and there exists  a homogeneous symmetric function $U(\x)$ of degree $n$ with coefficients in $\Z$ such that   
 $$u(q) =(q;q)_n\,\, \Ps (U(\x)).$$ If $dm =n$  then $u(\xi_d)$ is the coefficient of $z_{d^m}^{-1} p_{d^m}$ in the expansion of $U(\x)$ in the power sum basis.  
\end{lemma}

The following conjecture  generalizes Theorem~\ref{unity}.  Recall the definitions of $\mathcal N_{P,\mu}$ and $\widetilde{\mathcal N}_{P,\mu}$ given prior to Conjecture~~\ref{powercon1}  and Proposition~\ref{powercon2}, respectively.
\begin{con}  \footnote{See Section~\ref{newsec}: {\em Recent developments}}\label{genunity} Let $G=([n],E)$ be the incomparability graph of  a natural unit interval order $P$ and let
$$A_G(q,t):=\sum_{\sigma \in \sg_n} q^{\maj_P(\sigma)} t^{\inv_G(\sigma)}.$$  If $dm = n$ and  $\xi_d$ is any primitive $d$th root of unity then  
\begin{equation} \label{rooteq} A_G(\xi_d,t) = \sum_{\sigma \in \mathcal N_{P,d^m}} t^{\inv_G(\sigma)}= [d]_t^m \sum_{\sigma \in \widetilde{\mathcal N}_{P,d^m}} t^{\inv_G(\sigma)} .\end{equation} Moreover,  the palindromic polynomial $A_G(\xi_d,t)$ is   unimodal.
\end{con}

\begin{prop}\label{impliesunity} If Conjecture~\ref{powercon1} is true then Conjecture~\ref{genunity} is true.
\end{prop}

\begin{proof}   It follows from Lemma~\ref{thdes} and equation (\ref{stabprinceq}) that $A_G(\xi_d,t)$ is the coefficient of $z_{d^m}^{-1}p_{d^m}$ in the expansion of $\omega X_G(\x,t)$ in the power sum basis.  By Conjecture~\ref{powercon1} this coefficient is equal to $\sum_{\sigma \in \mathcal N_{P,d^m}} t^{\inv_G(\sigma)}$, which by Proposition~\ref{powercon2} equals $[d]_t^m \sum_{\sigma \in \widetilde{\mathcal N}_{P,d^m}} t^{\inv_G(\sigma)}$.
It is also asserted in Conjecture~\ref{powercon1} that  $\sum_{\sigma \in \mathcal N_{P,d^m}} t^{\inv_G(\sigma)}$ is unimodal.
\end{proof}

In addition  to the  cases in which Conjecture~\ref{powercon1} has been verified, Conjecture~\ref{genunity} is true for $d = 1$ and $d=n$.  Equation (\ref{rooteq}) is trivial in the $d=1$ case and the unimodality of $A_G(1,t)$ was established by De Mari, Procesi, and Shayman in \cite{DePrSh} and is discussed in Section~\ref{hesssec}.  For $d=n$ we have the following result.
\begin{thm} Let $G$ and $\xi_n$ be as in Conjecture~\ref{genunity} and let $b_i$ be as in (\ref{bi.eq}).  Then
$$A_G(\xi_n,t) = [n]_t \prod_{i=2}^n [b_i]_t.$$ 
Consequently $A_G(\xi_n,t)$ is palindromic and unimodal. \end{thm}
\begin{proof} By Lemma~\ref{thdes} and Theorem~\ref{speccor}, $A_G(\xi_n,t)$ is the coefficient of $n^{-1}p_n$ in the $p$-basis expansion of $\omega X_G(\x,t)$.  The result now follows from Theorem~\ref{pcoefth} and Propostion~\ref{tooluni}.
\end{proof}

\section{Hessenberg varieties} \label{hesssec}
In this section we discuss a conjectured connection, first presented in  \cite{ShWa3}, between the chromatic quasisymmetric functions and Hessenberg varieties.

Given a positive integer $n$, let $\mathcal F_n$ be the variety of all  flags of subspaces  $$ F: V_1 \subset V_2 \subset \cdots \subset V_n = \C^n$$ with $\dim V_i = i$.  Fix a diagonal $n \times n$ complex matrix $s$ with $n$ distinct eigenvalues. 
Let ${\bf m}$ be  a weakly increasing sequence $  (m_1,\dots,m_{n-1})$ of  integers in $[n]$ satisfying $m_i \ge i$ for each $i$.  The {\em regular semisimple Hessenberg variety of type $A$} associated to ${\bf m}$ is
\[
\hess ({\bf m}):=\{F \in \mathcal F_n \mid s V_i \subseteq V_{m_i}  \mbox{ for all } i \in [n-1]\}.
\]

The regular semisimple Hessenberg varieties of type A associated to $(r,r+1,\dots,n,\dots,n)$ were studied initially by De Mari and Shayman in \cite{DeSh}.  Hessenberg varieties for arbitrary Lie types were defined and studied by De Mari, Procesi and Shayman in \cite{DePrSh}.  Such varieties are determined by certain subsets of a root system.  In \cite[Theorem 11]{DePrSh} it is noted that, for arbitrary Lie type, the Hessenberg variety associated with the set of (negatives of) simple roots is precisely the  toric variety associated with the corresponding Coxeter complex.  
In particular, in Lie type A, the Hessenberg variety associated with the simple roots is $\hess(2,3,\dots,n)$. Thus
\begin{equation}\label{torhess0}  \hess(2,3,\dots,n) = \v_n,\end{equation}
where $\v_n$ is the toric variety associated with the Coxter complex of type $A_{n-1}$.

The symmetric group acts naturally on the type A Coxeter complex and this action induces an action of $\sg_n$ on $\v_n$, which in turn induces  a representation of $\sg_n$ on its cohomology $H^*(\v_n)$.  Stanley \cite[Proposition 12]{St0}, using a recurrence of Procesi \cite{Pr}, obtained the  formula 
\begin{equation} \label{toreq} \sum_{n\ge 0}\sum_{j=0}^{n-1} \ch H^{2j}(\v_n) t^j z^n = {(1-t) H(z) \over H(zt) - t H(z)}, \end{equation} 
where $\ch$ is the Frobenius characteristic and   $H(z) := \sum_{n \ge 0} h_n z^n$.  (The cohomology of every $\hess(\bf m)$ vanishes in odd dimensions.)

Let $P$ be a natural unit interval order  on $[n]$ and let ${ \bf m}$ be such that  $P=P({\bf m})$ (cf. the paragraph before Proposition~\ref{natunitprop2}).  
It is convenient to define the Hessenberg variety associated to   $P$ by  $\hess (P):= \hess ({\bf m})$.
Hence (\ref{torhess0}) can be rewritten as 
\begin{equation}\label{torhess}  \hess(P_{n,2}) = \v_n.\end{equation}

For each natural unit interval order $P$ on $[n]$, Tymoczko   \cite{Ty3}, \cite{Ty2} has defined a representation of $\S_n$ on $H^*(\hess(P))$ via a theory of Goresky, Kottwitz and MacPherson, known as GKM theory (see \cite{GoKoMacP}, and see \cite{Ty} for an introductory description of GKM theory).
The centralizer 
$T$ of $\bft$ in $GL_n(\cc)$ consists of all nonsingular diagonal
matrices.
  The torus $T$ acts (by left translation) on $\hess (P)$.  The technical conditions required for application of GKM theory  are satisfied by this action, and it follows that one can describe the cohomology of $\hess(P)$ using the associated moment graph $M(P)$.  The vertex set of $M(P)$ is $\S_n$.  Given $v,w \in \S_n$, the edge $\{v,w\}$ lies in $E(M(P))$ if $v=(i,j)w$ for some transposition $(i,j)$ such that $\{i,j\} \in E(\inc(P)) $.  Thus the action of $\S_n$ on itself by right multiplication determines an action of $\S_n$ on $M(P)$, and this action can be used to define a linear representation of $\S_n$ on the cohomology $H^\ast(\hess (P))$.  For the remainder of the section, we will consider $H^\ast(\hess(P))$ to be the $\S_n$-module determined by  the representation defined by Tymoczko.

It turns out that $H^{*}(\hess (P_{n,2}))$ is isomorphic to the representation of $\S_n$ on $H^*(\v_n)$ given in (\ref{toreq}). Hence by (\ref{symEulereq}) 
$$\sum_{j\ge 0} \ch H^{2j}(\hess (P_{n,2})) t^j= \omega X_{G_{n,2}}(\x,t) .$$
In the case $r=1$,   $\hess(P_{n,r})$ consists of $n!$ isolated points and  $H^{0}(\hess(P))$ is the regular $\S_n$-module.  Hence 
$$\sum_{j\ge 0} \ch H^{2j}(\hess (P_{n,1})) t^j= h_1^n.$$
 When $r=n$, it follows from \cite[Proposition 4.2]{Ty2}  that  for all $j$,  the $\S_n$-module $H^{2j}(\hess(P))$  is  isomorphic to the direct sum of copies of the trivial $\S_n$-module, where the number of copies equals the number of permutations in $\sg_n$ with $j$ inversions. Hence by (\ref{majinveq}),
 $$\sum_{j\ge 0} \ch H^{2j}(\hess (P_{n,n})) t^j= \sum_{\sigma \in \sg_n}  t^{\inv(\sigma)} h_n = [n]_t! h_n.$$
 By comparing with Table 1 in Section~\ref{esec}, we conclude that for $r=1$ and $r=n$,
 $$\sum_{j\ge 0} \ch H^{2j}(\hess (P_{n,r})) t^j= \omega X_{G_{n,r}}(\x,t) .$$
This has also been verified  for $r=n-1$ and $r=n-2$ by comparing computations of Tymoczko (personal communication, see also \cite{Te2}) to Table 1.  
 
 \begin{con}  \footnote{See Section~\ref{newsec}: {\em Recent developments}}\label{hesschrom} Let $P$ be a natural unit interval order on $[n]$.
Then \begin{equation}\label{mainconeq}\sum_{j=0}^{|E(\inc(P))|} \ch H^{2j}(\hess(P)) t^j = \omega X_{\inc(P)}(\x,t).\end{equation}
\end{con}

 Combining Proposition \ref{disjointprop} with \cite[Theorem 4.10]{Te}, one sees that if Conjecture \ref{hesschrom} holds for all natural unit interval orders $P$ such that $\inc(P)$ is connected, then it holds in general.

Further evidence of the validity of this conjecture is obtained by taking the stable principal specialization of both sides of (\ref{mainconeq}).  Indeed,  by Theorem~\ref{speccor}, equation (\ref{mainconeq}) specializes to
 $$(q;q)_n\sum_{j=0}^{|E(\inc(P))|}  {\bf ps}(\ch H^{2j}(\hess(P))) t^j = A_{\inc(P)}(q,1,t).$$  
Now for any representation $V$ of $\S_n$, by evaluating the polynomial $(q;q)_n {\bf ps} (\ch V)$ at $q=1$, one obtains the dimension of $V$. 
Hence the specialization of (\ref{mainconeq}) at $q=1$ is the formula
 \begin{equation} \label{dmpseq} \sum_{j=0}^{|E(\inc(P))|} \dim H^{2j}(\hess(P)) t^j = A_{\inc(P)}(1,1,t).\end{equation}
of De Mari, Procesi and Shayman \cite{DePrSh}.

Equation (\ref{dmpseq}) is also equivalent to the statement that the coefficient of the monomial symmetric function $m_{1^n}$ in the expansion 
of the left hand side of (\ref{mainconeq}) in the monomial symmetric function basis equals the coefficient of 
$m_{1^n}$ in the expansion of the right hand side of (\ref{mainconeq}).  

It is also pointed out in \cite{DeSh, DePrSh}  that it  follows from  the hard Lefschetz theorem that $A_{\inc(P)}(1,1,t)$ is palindromic and unimodal, since  $\hess(P)$ is smooth. Hence Conjecture~ \ref{GEulcon}  holds when $p=q=1$.  As far as we know there is no elementary proof of this case of the conjecture  (cf. \cite[Ex. 1.50 f]{St1}).

\begin{prop} \label{uniconsprop} If Conjecture~\ref{hesschrom} is true then $X_{\inc(P)}(\x,t)$ is Schur-unimodal.  Consequently, Conjecture~\ref{hesschrom} implies Conjecture~\ref{GEulcon}. \end{prop}
\begin{proof} MacPherson and Tymoczko \cite{MacTy} show that the hard Lefschetz map for $H^*(\hess(P))$ commutes with the action of the symmetric group.  It follows that
$\sum_{j=0}^{|E(\inc(P))|} \ch H^{2j}(\hess(P)) t^j$ is  Schur-unimodal.  Hence Schur-unimodality of $X_{\inc(P)}(\x,t)$ would follow from (\ref{mainconeq}).  The consequence follows from Proposition~\ref{unispecprop}. \end{proof}

Conjecture~\ref{hesschrom}, if true,  would provide, via Theorem~\ref{schurcon}, a solution of the following problem.

\begin{prob}[Tymoczko \cite{Ty2}] \label{tymprob}   Describe the decomposition of the representation of $\S_n$ on $H^{2j}(\hess(P))$ into irreducibles.
\end{prob}

Although Conjecture~\ref{hesschrom} does not imply $e$-positivity of $X_{\inc(P)}(\x,t)$, it 
suggests an approach to establishing $e$-positivity.    Indeed, Conjecture \ref{stancon} follows from Conjecture~\ref{hesschrom} and the following conjecture. 

\begin{con} \label{tympermcon} For every natural unit interval order $P$, every  $\S_n$-module  $H^{2j}(\hess(P))$ is a permutation module  in which each point stabilizer is a Young subgroup.
\end{con}

\section{Recent developments} \label{newsec}

Athansiadis \cite{At} has proved our formula (\ref{pexpandeq}) in Conjecture~\ref{powercon1}, which gives the power-sum expansion of $\omega X_G(\x,t)$ when $G$ is the incomparability graph of a natural unit interval order.  As was discussed in Proposition~\ref{impliesunity},   formula (\ref {rooteq}) in Conjecture~\ref{genunity} is a consequence.

Brosnan and Chow have posted a proof of our Conjecture~\ref{hesschrom} (Conjecture~\ref{introhesschrom}), which 
relates the chromatic quasisymmetric functions to the Hessenberg varieties, see \cite{BrCh}.   As was discussed in 
Proposition~\ref{uniconsprop}, Schur-unimodality of $X_G(\x,t)$ and $(q,p)$-unimodality of $A_G(q,p,t)$ (Conjecture~\ref{GEulcon}) follow from this.

 More recently,  Guay-Paquet  has posted an alternative approach to proving our Conjecture~\ref{hesschrom} (Conjecture~\ref{introhesschrom}); see \cite{Gu2}.   Also Abe, Harada, Horiguchi, and Masuda have posted a formula for the coefficient of $s_n$ in the Schur expansion of the left hand side of (\ref{mainconeq}) that is identical to the right  hand side of (\ref{eschureq}), thereby providing an independent proof of the fact that the coefficients of $s_n$ on both sides of (\ref{mainconeq}) are equal, see \cite{AHHM}.
 
Clearman, Hyatt, Shelton and Skandera \cite{CHSS,CHSS2} have given an algebraic interpretation of the  chromatic quasisymmetric functions of incomparability graphs of natural unit interval orders in terms of  characters of type A Hecke algebras evaluated at Kazhdan-Lusztig basis elements.  This refines a known algebraic interpretation for chromatic symmetric functions developed in  Haiman \cite{Ha}, Stembridge \cite{Ste1}, Stanley-Stembridge \cite{StSte}, and Stanley \cite{St4}.

\appendix

\section{Permutation statistics} \label{qEulersec}
Let $\sigma \in \sg_n$. The 
excedance number of $\sigma$ is given by
$$\exc(\sigma) :=|\{i \in [n-1] : \sigma(i) > i\}|.$$   The descent set of $\sigma$  is
given by $$\Des(\sigma) :=  \{i \in [n-1] : \sigma(i) > \sigma(i+1)\},$$ and the descent number and 
major index are $$\des(\sigma) := |\Des(\sigma) |\,\, \mbox{  and }\,\,
 \maj(\sigma):= \sum_{i\in \Des(\sigma)} i .$$ 
 The inversion statistic is defined as $$\inv(\sigma) = |\{(i,j) \in [n] \times [n] : i < j \mbox{ and } \sigma(i) > \sigma(j)\}|.$$
 The number of fixed points of $\sigma$ is given by 
 $$\fix(\sigma):= |\{i\in [n] : \sigma(i) = i \}|.$$
 
 It is a well known result of MacMahon that the permutation statistics $
 \exc$ and $\des$ are   equidistributed on $\S_n$.  The common generating functions for these statistics are called Eulerian polynomials.  That is, the Eulerian polynomials are defined as 
\bq \label{defEuler} A_{n}(t) :=\sum_{\sigma \in \S_n}t^{\des(\sigma)} = \sum_{\sigma \in \S_n}t^{\exc(\s)} .\eq
 Any permutation statistic with generating function $A_n(t)$ is called an Eulerian statistic. 

It is also well-known (and due to MacMahon) that the  major index is equidistributed with the inversion statistic and that
\begin{equation} \label{majinveq} \sum_{\sigma \in \S_n}q^{\maj(\sigma)} = \sum_{\sigma \in \S_n}q^{\inv(\s)} = [n]_q!,\end{equation}
where 
$$[n]_q := 1+q+ \dots + q^{n-1} \mbox{ and } [n]_q! := [n]_q [n-1]_q \dots [1]_q. $$
Any permutation statistic equidistributed with $\maj $ and $\inv$ is called a Mahonian permutation statistic.  
 
In \cite{ShWa1,ShWa}, we studied the joint distribution of the (Mahonian) major index, and the (Eulerian) excedance number.   We showed that the $(q,r)$-Eulerian polynomial defined by 
$$A_n(q,r,t):=\sum_{\sigma \in \S_n} q^{\maj(\sigma)-\exc(\sigma)}  r^{\fix(\sigma)} t^{\exc(\sigma)}$$ has exponential generating function given by 
 \begin{equation} \label{expgeneq}
1+\sum_{n \geq 1}A_n(q,r,t)\frac{z^n}{[n]_q!}=\frac{(1-t)\exp_q(rz)}{\exp_q(tz)-t\exp_q(z)},
\end{equation}
where $$\exp_q(z) := \sum_{n \ge 0} \frac {z^n}{[n]_q!}.$$
By setting $q=1$ and $r=1$ in  (\ref{expgeneq}) one gets Euler's classical exponential generating function formula for the Eulerian polynomials $A_n(t)= A_n(1,1,t)$.

\section{Palindromicity and unimodality} \label{palin.appen}
 
 We say that a polynomial $f(t) = \sum_{i=0}^{n} a_i t^i$ in $\Q[t]$ is positive, unimodal and palindromic with center of symmetry $\frac n 2$ if 
$$ 0 \le a_{0} \le a_{1} \le  \dots \le a_{\lfloor {n\over 2} \rfloor} = a_{\lfloor {n+1\over 2} \rfloor} \ge \dots \ge a_{n-1}\ge a_{n} \ge 0,$$
and $a_i = a_{n-i}$ for all $i=0,\dots,n$.
The next  proposition gives a basic tool for establishing unimodality of polynomials in $\Q[t]$.
\begin{prop}[see {\cite[Proposition 1]{St0}}] \label{tooluni} Let $A(t)$ and $B(t)$ be positive,  unimodal and palindromic polynomials in $\Q[t]$ with respective centers of symmetry $c_A$ and $c_B$.  Then
\begin{enumerate}
\item $A(t) B(t) $ is positive,  unimodal, and palindromic  with center of symmetry $c_A +c_B$.
\item If $c_A=c_B$ then $A(t) +B(t)$ is positive,  unimodal and palindromic with center of symmetry $c_A$.
\end{enumerate}
\end{prop} 

For example, from Proposition~\ref{tooluni} (1) we see that   each term of the sum $[5]_t [2]_t + [6]_t+ t[4]_t+[4]_t [3]_t $ is positive, unimodal and palindromic with center of symmetry $\frac 5 2$. From Proposition~\ref{tooluni} (2) we conclude that the sum itself is also positive, unimodal and palindromic with center of symmetry $\frac 5 2$.  Note that the well-known palindromicity and unimodality of each row of Pascal's triangle  is an immediate consequence of Proposition~~\ref{tooluni} (1); indeed $(1+t)^n$ is a  product of positive, unimodal and palindromic polynomials.

We say that a polynomial $f(q) \in \Q[q] $ is $q$-positive if its nonzero coefficients are positive. Given polynomials $f(q),g(q) \in \Q[q]$ we say that $f(q) \le_q g(q)$ if $g(q) - f(q)$ is $q$-positive.  
More generally, let $R$ be an algebra over $\Q$ with basis $b$.  An element $s\in R$ is said to be $b$-{\em positive} if the expansion of $s$ in the basis $b$ has nonnegative coefficients.  Given $r,s \in R$, we say that $r \le_b s$ if $s-r$ is $b$-positive. 
 If $R = \Q[q]$ and $b = 
\{q^i: i \in \N\}$ then $b$-positive is  what we called $q$-positive above and $<_b$ is the same as $<_q$.

\begin{definition}\label{paldef} Let $R$ be a $\Q$-algebra with basis $b$.   We say that a polynomial
$A(t):=a_0 + a_1 t +\cdots  + a_{n} t^{n}\in R[t]$ is $b$-{\em positive} if each coefficient $a_i$ is $b$-positive.   The  polynomial $A(t)$ is $b$-{\em unimodal} if
$$a_0 \le_b a_1\le_b \cdots \le_b a_c \ge_b a_{c+1} \ge_b a_{c+2} \ge_b \cdots \ge_b a_n,$$ for some $c$.  The polynomial $A(t)$ is {\em palindromic} with center of symmetry $\frac n 2$ if $a_j = a_{n-j}$  for $0 \le j \le n$.
\end{definition} 
Proposition~\ref{tooluni} holds for polynomials $A(t), B(t) \in R[t]$, with positivity and unimodality replaced by $b$-positivity and $b$-unimodality, respectively.  

Note that if $b$ is a basis for the $\Q$-algebra $R$ then it is also a basis for the $\Q[t]$-algebra $R[t]$.   Indeed,    let $A(t) = a_0 + a_1 t +\cdots  + a_{n} t^{n} \in R[t]$ and let  $a_{i,j}\in \Q$ be the expansion coefficients of each $a_j$ in the basis  $b:= \{b_i : i \ge 0\}$;  that is, $a_j = \sum_{i \ge 0} a_{i,j} b_i$.  We have
$$A(t) = \sum_{j=0}^n \left( \sum_{i \ge 0} a_{i,j} b_i \right) t^j = \sum_{i \ge 0}\left( \sum_{j=0}^n a_{i,j}  t^j \right) b_i .$$

The next result is a convenient reformulation of Proposition~\ref{tooluni} (2)  for more general coefficient rings.  

\begin{prop} \label{tooluni2}
Let  $R$ be a $\Q$-algebra with basis $b$. Suppose $P(t) := \sum_{i =0}^m X_i Y_i(t)  \in R[t]$, where $X_i \in R\setminus \{0\}$ is $b$-positive and $Y_i(t) \in \Q[t]\setminus\{0\}$.  Then the polynomial $P(t)$  is $b$-positive,  $b$-unimodal, and palindromic  with center of symmetry $c$ if  each   $Y_i(t)$  is positive,  unimodal,  and palindromic with center of symmetry $c$.
\end{prop}

\begin{proof}Since $Y_i(t)$  is positive,  unimodal,  and palindromic with center of symmetry $c$, we have  $Y_i(t) = \sum_{j=0}^{2c} y_{i,j} t^j$,    where $y_{i,j} \ge 0$, $y_{i,j} = y_{i,2c-j}$ for each $i,j$, and
$$y_{i,0} \le y_{i,1} \le \dots \le y_{i,\lfloor c \rfloor},$$
for each $i$.  Now 
\begin{eqnarray*} P(t) &=& \sum_{i =0}^m \sum_{j=0}^{2c}X_i \,y_{i,j}\, t^j \\
&=& \sum_{j=0}^{2c} \left (\sum_{i=0}^m y_{i,j} X_i \right ) t^j.\end{eqnarray*} 
Since each $y_{i,j} \ge 0$ and each $X_i$ is $b$-positive, each $\sum_{i=0}^m y_{i,j} X_i$ is $b$-positive, making $P(t)$ $b$-positive.  
Since  $y_{i,j} = y_{i,2c-j}$, each $\sum_{i=0}^m y_{i,j} X_i = \sum_{i=0}^m y_{i,2c-j} X_i $, making $P(t)$ palindromic. 
 Finally, since $y_{i,j+1} - y_{i,j} \ge 0$ for each $j=0,\dots, \lfloor c \rfloor - 1$, we see that $$\sum_{i=0}^m y_{i,j+1} X_i - \sum_{i=0}^m y_{i,j} X_i = \sum_{i=0}^m (y_{i,j+1}-y_{i,j}) X_i $$ 
is $b$-positive, making $P(t)$ $b$-unimodal. 
 \end{proof}

For example, consider the polynomial  $$P(q,t) := ([5]_t [2]_t + [6]_t + t[4]_t)(3q+q^5) + [4]_t [3]_t (1+2q^3).$$  By Proposition~\ref{tooluni} the polynomials $[5]_t [2]_t + [6]_t + t[4]_t$ and  $[4]_t [3]_t $ in $\Q[t]$ are positive, unimodal and palindromic  with center of symmetry $\frac 5 2$.  The polynomials   $3q+q^5$ and $1+2q^3$ are $q$-positive.  It therefore follows from Proposition~\ref{tooluni2} that $P(q,t)\in \Q[q][t]$ is $q$-positive, $q$-unimodal and palindromic with center of symmetry $\frac 5 2$.

\begin{remark}We have changed the terminology from our previous papers (e.g. \cite{ShWa}) where we referred to $q$-positivity and $q$-unimodality of polynomials in $\Q[q][t]$ as $t$-positivity and $t$-unimodality, respectively.
\end{remark}

\section{Unimodality of $q$-Eulerian polynomials and Smirnov word enumerators}  \label{app.eul}
 \subsection{The $q$-Eulerian polynomial $A_n(q,t)$}It is well known that the Eulerian polynomials are palindromic and unimodal.  In \cite{ShWa} a q-analog of this result is obtained for $A_n(q,t):=A_n(q,1,t)$.  In this section we present an alternative proof as an easy consequence of the closed form formula  given in the following result.

 \begin{thm}\label{newth}
  For $n \ge 1$,
\begin{equation}\label{newformq} A_n(q,t)= \sum_{m=1}^{\lfloor {n+1 \over 2} \rfloor}   \sum_{\substack{
k_1,\dots, k_m \ge 2 \\ \sum k_i = n+1}}\,\,\, \left[ \begin{array}{c} n \\ k_1-1,  k_2, \dots, k_m \end{array} \right]_q \,\,\,t^{m-1} \prod_{i=1}^m  [k_i-1]_{t} \,\,,\end{equation}
where $$\left[\begin{array}{c} n \\k_1,\dots,k_m\end{array}\right]_q = {[n]_q!
\over [k_1]_q!\cdots [k_m]_q!}.$$
 \end{thm}
 
\begin{proof} We rewrite  the $r=1$ case of (\ref{expgeneq})  as
$$ 1+\sum_{n \geq 1}A_n(q,t)\frac{z^n}{[n]_q!} = 1 + \frac {\sum_{n \ge 1} [n]_t \frac{z^n}{[n]_q!}} {1-t\sum_{n\ge 2} [n-1]_t \frac{z^n}{[n]_q!} }
$$ and extract the coefficients of $\frac{z^n}{[n]_q!}$.
\end{proof}

\begin{cor}[Shareshian and Wachs \cite{ShWa}] \label{qunicor} For all $n \ge 1$, the polynomial $A_n(q,t)$ is $q$-unimodal and palindromic.
\end{cor}   

\begin{proof}   By Proposition~\ref{tooluni} (1), the polynomial $t^{m-1} \prod_{i=1}^m   [k_i-1]_{t}
$ is positive, unimodal and palindromic with center of symmetry
$$m-1+\sum_{i=1}^m \frac {k_i-2} 2 =  \frac {n-1} 2.$$
Since each polynomial  $\left[ \begin{array}{c} n \\ k_1-1,  k_2, \dots, k_m \end{array} \right]_q $ is $q$-positive, the result now follows from Proposition~\ref{tooluni2}.
\end{proof}

\subsection{The Smirnov word enumerator $W_n(\x,t)$}  \label{smirsec}
Let $W_n$ be the set of all words of length $n$ over alphabet $\PP$ with no adjacent repeats, i.e.,
$$W_n := \{w \in \PP^n : w_i\ne w_{i+1}\,\, \forall i = 1,2,\dots,n-1\}. $$  These words are sometimes called {\em Smirnov words}.  Define the enumerator
$$W_n(\x,t):= \sum_{w\in  W_n} t^{\des(w)}x_{w_1} \cdots x_{w_n} ,$$ where
 $\des(w)= \{i \in [n-1] : w_i > w_{i+1}\}$.  
 
 Stanley (personal communication) observed that by $P$-partition reciprocity \cite[Section 4.5]{St1}, $W_n(\x,t)$ is equal to  $\sum_{j=0}^n  \omega Q_{n,j}(\x) t^j$, where  $Q_{n,j}(\x)$ is the  Eulerian quasisymmetric function considered in \cite{ShWa}.  This observation follows from the characterization of $Q_{n,j}(\x)$ involving ``banners", given in \cite[Theorem 3.6]{ShWa}.
The following result, which refines a formula of  Carlitz, Scoville and Vaughan \cite{CaScVa},  is therefore equivalent to the $r=1$ case of    \cite[Theorem 1.2]{ShWa}.
\begin{thm}[Stanley (see {\cite[Theorem  7.2]{ShWa}})] \label{stanth}  Let $$E(z) := \sum_{n \ge 0} e_n(\x) z^n.$$  Then

\begin{equation}\label{carlg} \sum_{n\ge 0} W_n(\x,t) z^n = {(1-t) E(z) \over E(zt) - t E(z)}. \end{equation}  
\end{thm}

 The following closed form formula for $W_n(\x,t)$ is obtained by rewriting (\ref{carlg}) as 
\begin{equation} \label{recarlg} \sum_{n\ge 0} W_n(\x,t) z^n = 1+ \frac{\sum_{n\ge 1} [n]_t e_n z^n} { 1-t \sum_{n\ge 2} [n-1]_t e_n z^n}.\end{equation}
\begin{thm} \label{newsymth} For all $n\ge 1$,
$$W_n(\x,t) = \sum_{m=1}^{\lfloor {n+1 \over 2} \rfloor}   \sum_{\substack{
k_1,\dots, k_m \ge 2 \\ \sum k_i = n+1}}\,\,\, e_{ (k_1-1,  k_2, \dots, k_m)} \,\,\,t^{m-1} \prod_{i=1}^m  [k_i-1]_{t} \,\,$$
where $e_{(i_1,\dots,i_m)} :=  e_{i_1}\dots e_{i_m}$ for all $(i_1,\dots,i_m) \in \pp^m$.
\end{thm}

The following consequence of Theorem~\ref{newsymth} is obtained by applying Propositions~\ref{tooluni} 
and~\ref{tooluni2}.

\begin{cor} \label{symunicor} For all $n\ge 1$, the polynomial $W_n(\x,t)$ is $e$-positive,   $e$-unimodal and  palindromic.
\end{cor} 
 
\begin{remark} Our proof of Corollary~\ref{symunicor} is based on an observation of Haiman in \cite[Equation (4.9)]{Ha}.   Stanley  observed in \cite{St0} that the weaker result that the coefficient of  $z^n$ in the formal power series $(1-t)E(z) \over E(zt)-tE(z)$ is a palindromic Schur-unimodal polynomial  is a consequence of (\ref{toreq}) and the hard Lefshetz theorem. 
Stembridge observed in \cite{Ste1}  that p-unimodality of the coefficients is a also a consequence of (\ref{toreq})  and the hard Lefshetz theorem.\end{remark}

\begin{remark} Theorem~\ref{newth} and Corollary~\ref{qunicor} are specializations of    Theorem~\ref{newsymth} and Corollary~\ref{symunicor}, respectively (see Section~\ref{specsec}).
\end{remark}

\section*{Acknowledgments}  

We are grateful  to Richard Stanley for informing us of Theorem~\ref{stanth}, which led us to  introduce the  chromatic quasisymmetric functions.   Darij Grinberg suggested a large number of corrections and improvements to exposition after reading an earlier version of this paper, for which we are also grateful. We thank  Julianna Tymoczko for useful
discussions.  Some of the work on this paper was done while the second author was a Visiting Scholar at the University of California, Berkeley and a member of MSRI.  The author thanks these institutions for their hospitality.

\end{document}